\documentclass[11pt,a4paper]{article}

\usepackage{amsmath}
\usepackage{amsthm}
\usepackage{amssymb}
\usepackage{cain_arxiv}
\usepackage{graphicx}

\theoremstyle{definition}
\newtheorem{definition}{Definition}[section]

\newtheorem{pumprule}{Pumping rule}

\newtheorem{arrowrule}{Arrow rule}

\theoremstyle{plain}
\newtheorem{proposition}[definition]{Proposition}
\newtheorem{lemma}[definition]{Lemma}
\newtheorem{theorem}[definition]{Theorem}
\newtheorem{corollary}[definition]{Corollary}

\numberwithin{equation}{section}

\def\fullref#1#2{%
  \ifdefined\hyperref%
    {\hyperref[#2]{#1 \penalty 200\relax\ref*{#2}}}%
  \else%
    {#1 \penalty 200\relax\ref{#2}}%
  \fi%
}

\newcommand{\defterm}[1]{\textit{#1}}

\newcommand{\fsa}[1]{\mathfrak{#1}}

\newcommand{\nset}{\mathbb{N}}
\newcommand{\zset}{\mathbb{Z}}

\newcommand{\emptyword}{\varepsilon}

\newcommand{\elt}[1]{\overline{#1}}

\DeclareMathOperator{\rr}{\mathnormal{\mu}}
\DeclareMathOperator{\rre}{\mathnormal{\overline{\mu}}}

\newcommand{\conv}{\mathrm{conv}}

\newcommand{\inpargraphic}[1]{\par\medskip\centerline{\includegraphics{#1}}\medskip\noindent\ignorespaces}

\def\ext#1{\overline{#1}}

\begin{document}

\title{Unary FA-pre\-sent\-a\-ble binary relations: transitivity and classification results}
\author{Alan J. Cain \& Nik Ru\v{s}kuc}
\date{}

\thanks{The first author's research was funded by the European
  Regional Development Fund through the programme {\sc COMPETE} and by
  the Portuguese Government through the {\sc FCT} (Funda\c{c}\~{a}o
  para a Ci\^{e}ncia e a Tecnologia) under the project {\sc
    PEst-C}/{\sc MAT}/{\sc UI0}144/2011 and through an {\sc FCT}
  Ci\^{e}ncia 2008 fellowship. Two visits by the first author to the
  University of St Andrews, where much of the work described in this
  paper was carried out, were supported by the {\sc EPSRC}-funded
  project {\sc EP}/{\sc H0}11978/1 `Automata, Languages, Decidability
  in Algebra'.}

\maketitle

\address[AJC]{%
Centro de Matem\'{a}tica, Faculdade de Ci\^{e}ncias, Universidade do Porto, \\
Rua do Campo Alegre 687, 4169--007 Porto, Portugal
}
\email{%
ajcain@fc.up.pt
}
\webpage{%
www.fc.up.pt/pessoas/ajcain/
}

\address[NR]{%
School of Mathematics and Statistics, University of St Andrews, \\
North Haugh, St Andrews, Fife KY16 9SS, United Kingdom
}
\email{%
nik@mcs.st-andrews.ac.uk
}
\webpage{%
turnbull.mcs.st-and.ac.uk/~nik/
}

\begin{abstract}
Automatic presentations, also called FA-pre\-sent\-a\-tions, were
introduced to extend finite model theory to infinite structures whilst
retaining the solubility of fundamental decision problems. A
particular focus of research has been the classification of those
structures of some species that admit FA-pre\-sent\-a\-tions. Whilst some
successes have been obtained, this appears to be a difficult problem
in general. A restricted problem, also of significant interest, is to
ask this question for unary FA-pre\-sent\-a\-tions: that is,
FA-pre\-sent\-a\-tions over a one-letter alphabet. This paper studies
unary FA-pre\-sent\-a\-ble binary relations.

It is proven that transitive closure of a unary FA-pre\-sent\-a\-ble
binary relation is itself unary
FA-pre\-sent\-a\-ble. Characterizations are then given of unary
FA-pre\-sent\-a\-ble binary relations, quasi-orders, partial orders,
tournaments, directed trees and forests, undirected trees and forests,
and the orbit structures of unary FA-pre\-sent\-a\-ble partial and
complete mappings, injections, surjections, and bijections.
\end{abstract}

\section{Introduction}

Automatic presentations, also known as FA-pre\-sent\-a\-tions, were
introduced by Khoussainov \& Nerode \cite{khoussainov_autopres} to
fulfill a need to extend finite model theory to infinite structures
while retaining the solubility of interesting decision problems. They
have been applied to structures such as orders
\cite{khoussainov_autopo,khoussainov_linearorders,delhomme_autopresordinal}
and algebraic structures
\cite{oliver_autopresgroups,nies_rings,cort_apsg}.

One main avenue of research has been the classification of those
structures of some species that admit
FA-pre\-sent\-a\-tions. Classifications are known for finitely generated
groups \cite[Theorem~6.3]{oliver_autopresgroups} and cancellative
semigroups \cite[Theorem~13]{cort_apsg}, for integral domains (and
more generally for rings with identity and no
zero divisors) \cite[Corollary~17]{nies_rings}, for Boolean
algebras \cite[Theorem~3.4]{khoussainov_richness}, and for ordinals
\cite{delhomme_autopresordinal}.

Broadly speaking, there has been more success in classifying algebraic
structures than for combinatorial ones. However, it has been possible to
classify certain combinatorial structures that admit \emph{unary} FA-pre\-sent\-a\-tions
(that is, FA-pre\-sent\-a\-tions over a one-letter alphabet),
including, for example, bijective functions
\cite[Theorem~7.12]{blumensath_diploma}, equivalence relations
\cite[Theorem~7.13]{blumensath_diploma}, linear orders
\cite[Theorem~7.15]{blumensath_diploma}, and graphs
\cite[Theorem~7.16]{blumensath_diploma}.

In this paper, we study unary FA-presentable binary relations using a
new diagrammatic representation for unary FA-pre\-sent\-a\-tions that
we recently developed. We have already successfully deployed this
representation in a study of unary FA-presentable algebras
\cite{cr_subalg}. This representation allows us to visualize and
manipulate elements of a unary FA-presentable relational structure in
a way that is more accessible than the corresponding arguments using
languages and automata.

We being by describing the diagrammatic representation of unary
FA-pre\-sent\-a\-tions in \fullref{\S}{sec:pumping}. We first apply it
to prove that, given a unary FA-pre\-sent\-a\-ble structure with a
binary relation $R$, then that structure augmented by the transitive
closure of $R$ also admits a unary FA-pre\-sent\-a\-tion
(\fullref{Theorem}{thm:transclosure}). This result is peculiar to
\emph{unary} FA-pre\-sent\-a\-ble binary relations, because there
exist (non-unary) FA-pre\-sent\-a\-ble graphs in which reachability is
undecidable \cite[Examples~2.4(iv) \&~3.17]{rubin_survey}. As a
corollary of this stronger result, one recovers the previously-known
result asserting the regularity of the reachability relation for
unary FA-presentable undirected graphs of finite degree
\cite[Corollary~7.6]{khoussainov_unarygraphalgorithmic}.

We then turn to classification results. First, we give a technical
theorem that classifies all unary FA-pre\-sent\-a\-ble binary
relations (\fullref{Theorem}{thm:binrelchar}). With the aid of some
lemmata, this allows us to classify the unary FA-pre\-sent\-a\-ble
quasi-orders (\fullref{Theorem}{thm:qosetchar}), partial orders
(\fullref{Theorem}{thm:posetchar}), and tournaments
(\fullref{Theorem}{thm:tournamentchar}).

We then classify the unary FA-pre\-sent\-a\-ble directed trees
(\fullref{Theorem}{thm:directedtreechar}) and forests
(\fullref{Theorem}{thm:directedforestchar}) and then undirected trees
and forests (\fullref{Theorem}{thm:forestchar}). Finally, we classify,
in terms of their the orbit structures, unary FA-pre\-sent\-a\-ble
mappings (\fullref{Theorem}{thm:mapchar}), injections
(\fullref{Theorem}{thm:injectionchar}), and surjections
(\fullref{Theorem}{thm:surjectionchar}), and recover the
previously-known characterization of unary FA-presentable bijections
(\fullref{Theorem}{thm:bijectionchar}). These results also
characterize the unary FA-presentable partial versions of these types
of maps.

\section{Preliminaries}

The reader is assumed to be familiar with the theory of finite
automata and regular languages; see \cite[Chs~2--3]{hopcroft_automata}
for background reading. The empty word (over any alphabet) is denoted
$\emptyword$.

\begin{definition}
Let $L$ be a regular language over a finite alphabet $A$. Define, for $n \in \nset$,
\[
L^n = \{(w_1,\ldots,w_n) : w_i \in L\text{ for $i=1,\ldots,n$}\}.
\]
Let $\$$ be a new symbol not in $A$. The mapping $\conv : (A^*)^n\to ((A\cup\{\$\})^n)^*$ is defined as follows. Suppose
\begin{align*}
w_1 &= w_{1,1}w_{1,2}\cdots w_{1,m_1},\\
w_2 &= w_{2,1}w_{2,2}\cdots w_{2,m_2},\\
&\vdots\\
w_n &= w_{n,1}w_{n,2}\cdots w_{n,m_n},
\end{align*}
where $w_{i,j} \in A$. Then $\conv(w_1,\ldots,w_n)$ is defined to be
\[
(w_{1,1},w_{2,1},\ldots,w_{n,1})(w_{1,2},w_{2,2},\ldots,w_{n,2})\cdots (w_{1,m},w_{2,m},\ldots,w_{n,m}),
\]
where $m = \max\{m_i:i=1,\ldots,n\}$ and with $w_{i,j} = \$$ whenever
$j > m_i$.
\end{definition}

Observe that the mapping $\conv$ maps an $n$-tuple of words to a word of $n$-tuples.

\begin{definition}
Let $A$ be a finite alphabet, and let $R \subseteq (A^*)^n$ be a
relation on $A^*$. Then the relation $R$ is said to be \defterm{regular} if
\[
\conv R = \{\conv(w_1,\ldots,w_n) : (w_1,\ldots,w_n) \in R\}
\]
is a regular language over $(A\cup\{\$\})^n$.
\end{definition}

\begin{definition}
Let $\mathcal{S}=(S,R_1,\ldots,R_n)$ be a relational structure. Let
$L$ be a regular language over a finite alphabet $A$, and let $\phi :
L\rightarrow S$ be a surjective mapping. Then $(L,\phi)$ is an
\defterm{automatic presentation} or an \defterm{FA-pre\-sent\-a\-tion}
for $\mathcal{S}$ if, for all relations $R \in \{=,R_1,\ldots,R_n\}$,
the relation
\[
\Lambda(R,\phi)=\{(w_1,w_2,\ldots,w_r)\in L^r : R(w_1\phi,\ldots,w_{r}\phi)\},
\]
where $r$ is the arity of $R$, is regular.

If $\mathcal{S}$ admits an FA-pre\-sent\-a\-tion, it is said to be
\defterm{FA-pre\-sent\-a\-ble}.

If $(L,\phi)$ is an FA-pre\-sent\-a\-tion for $\mathcal{S}$ and the
mapping $\phi$ is injective (so that every element of the structure
has exactly one representative in $L$), then $(L,\phi)$ is said to be
\defterm{injective}.

If $(L,\phi)$ is an FA-pre\-sent\-a\-tion for $\mathcal{S}$ and $L$ is
a language over a one-letter alphabet, then $(L,\phi)$ is a
\defterm{unary} FA-pre\-sent\-a\-tion for $\mathcal{S}$, and
$\mathcal{S}$ is said to be \defterm{unary FA-pre\-sent\-a\-ble}.
\end{definition}

Every FA-presentable structure admits an injective \defterm{binary}
FA-presentation; that is, where the language of representatives is
over a two-letter alphabet; see
\cite[Corollary~4.3]{khoussainov_autopres} and
\cite[Lemma~3.3]{blumensath_diploma}. Therefore the class of binary
FA-presentable structures is simply the class of FA-presentable
structures. However, there are many structures that admit
FA-presentations but not unary FA-presentations: for instance, any
finitely generated virtually abelian group is FA-presentable
\cite[Theorem~8]{oliver_autopresgroups}, but unary FA-presentable
groups must be finite \cite[Theorem~7.19]{blumensath_diploma}. Thus
there is a fundamental difference between unary FA-presentable
structures and all other FA-presentable
structures.

\begin{definition}
If $(L,\phi)$, where $L \subseteq a^*$, is an injective unary
FA-pre\-sent\-a\-tion for a structure $\mathcal{S}$, and $s$ is an
element of $\mathcal{S}$, then $\ell(s)$ denotes the length of the unique
word $w \in L$ with $w\phi = s$. [Notice that $a^{\ell(s)} =
s\phi^{-1}$ for all elements $s$ of $\mathcal{S}$.]
\end{definition}

The fact that a tuple of elements $(s_1,\ldots,s_n)$ of a structure $\mathcal{S}$
satisfies a first-order formula $\theta(x_1,\ldots,x_n)$ is denoted $\mathcal{S}
\models \theta(s_1,\ldots,s_n)$. 

\begin{proposition}[{\cite[Theorem~4.4]{khoussainov_autopres}}]
\label{prop:firstorderreg}
Let $\mathcal{S}$ be a structure with an FA-pre\-sent\-a\-tion $(L,\phi)$. For
every first-order formula $\theta(x_1,\ldots,x_n)$ over the structure, the relation
\[
\Lambda(\theta,\phi) = \big\{(w_1,\ldots,w_n) \in L^n : \mathcal{S} \models\theta(w_1\phi,\ldots,w_n\phi)\big\}
\]
is regular.
\end{proposition}

Proposition~\ref{prop:firstorderreg} is fundamental to the theory
of FA-pre\-sent\-a\-tions and will be used without explicit reference
throughout the paper.

The following important result shows that in the case of unary
FA-presentations for infinite structures, we can assume that the
language of representatives is the language of \emph{all} words over a
one letter alphabet, and also that the map into the domain of the structure
is injective:

\begin{theorem}[{\cite[Theorem~3.1]{crt_unaryfa}}]
\label{thm:allwords}
Let $\mathcal{S}$ be an infinite relational structure that admits a unary
FA-pre\-sent\-a\-tion. Then
$\mathcal{S}$ has an injective unary FA-pre\-sent\-a\-tion
$(a^*,\psi)$.
\end{theorem}

We now gather some miscellaneous preliminary results that we will use
later in the paper:

\begin{theorem}[{\cite[Theorem~7.13]{blumensath_diploma}}]
\label{thm:equivrelchar}
Let $X$ be a set and $\rho$ an equivalance relation on $X$. Then
$(X,\rho)$ is unary FA-pre\-sent\-a\-ble if and only if there are
only finitely many infinite $\rho$-equivalence classes and there is a
bound on the cardinality of the finite $\rho$-equivalence classes.
\end{theorem}

The \defterm{disjoint union} of a family of structures $\mathcal{S}^{(i)} =
(S^{(i)},\sigma_1^{(i)},\ldots,\sigma_n^{(i)})$ with the same
signature (where $i$ ranges over an index set $I$) is the structure
\[
\bigl(\bigsqcup_{i \in I} S^{(i)},\bigsqcup_{i \in I} \sigma_1^{(i)},\ldots,\bigsqcup_{i \in I} \sigma_n^{(i)}\bigr),
\]
where $\sqcup$ denotes disjoint union as sets.

\begin{lemma}[{\cite[Proposition~7.6(ii)]{blumensath_diploma}}]
\label{lem:union}
The disjoint union of two unary FA-pre\-sent\-a\-ble structures with the
same signature is unary FA-pre\-sent\-a\-ble.
\end{lemma}

\begin{lemma}
\label{lem:countableunion}
The disjoint union of countably many isomorphic copies of a finite
structure is unary FA-pre\-sent\-a\-ble.
\end{lemma}

\begin{proof}
Let $\mathcal{S}$ be a finite structure. Suppose the domain $S$ of
$\mathcal{S}$ contains $n$ elements. Then there is a bijection $\psi :
\{a^0,\ldots,a^{n-1}\} \to S$. For any relation $\sigma$ of
$\mathcal{S}$, the relation $\Lambda(\sigma,\psi)$ is finite and thus
regular.

Define a map $\phi$ from $a^*$ to the (domain of) the disjoint union
of countably many copies of $\mathcal{S}$ by letting $a^{ni+j}\phi$ be
the element corresponding to $a^j\psi$ in the $i$-th copy of
$\mathcal{S}$, where $0 \leq j < n$. Then for any relation $\sigma$ of the disjoint union,
\[
\Lambda(\sigma,\phi) = (a^n,a^n)^*\Lambda(\sigma,\psi),
\]
and so is regular. Thus $(a^*,\phi)$ is a unary FA-pre\-sent\-a\-tion for the
the disjoint union of countably many copies of $\mathcal{S}$
\end{proof}

\section{Pumping and diagrams}
\label{sec:pumping}

This section develops a diagrammatic representation for unary
FA-pre\-sent\-a\-tions. Although we only discuss how this
representation works for binary relations, it also applies more
generally to unary FA-presentations for arbitrary relational structures;
see \cite[\S~4]{cr_subalg} for details.

Let $\mathcal{S}$ be a structure with a binary relation $R$, and
suppose $\mathcal{S}$ is unary FA-presentable. By
\fullref{Theorem}{thm:allwords}, there is an injective unary
FA-presentable structure $(a^*,\phi)$ for $\mathcal{S}$. Let $\fsa{A}$
be a deterministic $2$-tape synchronous automaton recognizing
$\Lambda(R,\phi)$. Let us examine the structure of the automaton
$\fsa{A}$. For ease of explanation, view $\fsa{A}$ as a directed graph
with no failure states: $\fsa{A}$ fails if it is in a state and reads
a symbol that does not label any outgoing edge from that state.

Since $\fsa{A}$ recognizes words in $\conv((a^*)^2)$, it will
only successfully read words lying in $(a,a)^*\bigl((a,\$)^* \cup
(\$,a)^*\bigr)$. Thus an edge labelled by $(a,a)$ leads to a state
whose outgoing edges can have labels $(a,a)$, $(a,\$)$,
$(\$,a)$. However, an edge labelled by $(a,\$)$ leads to a state all
of whose outgoing edges are labelled by $(a,\$)$. In fact, the
determinism of $\fsa{A}$ ensures there is at most one such outgoing
edge. Similarly, an edge labelled by $(\$,a)$ leads to a state with
either no outgoing edges or a single outgoing edge labelled by
$(\$,a)$.

Since $\fsa{A}$ is deterministic, while it successfully reads pairs
$(a,a)$ it follows a uniquely determined path which, if the string of
such pairs is long enough, will form a uniquely determined loop. This
loop, if it exists, is simple. From various points along this loop and
the path leading to it, paths labelled by $(a,\$)$ and $(\$,a)$ may
`branch off'. In turn, these paths, if they are long enough, lead into
uniquely determined simple loops. \fullref{Figure}{fig:unaryautomaton}
shows an example.

\begin{figure}[tb]
\centerline{\includegraphics{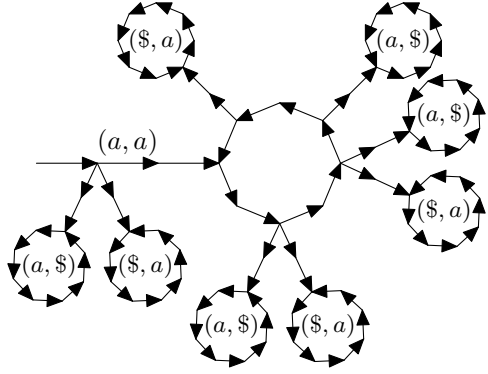}}
\caption{Example two-tape automaton recognizing
  $\Lambda(R,\phi)$. Edges labelled $(a,a)$ form a path that leads
  into a uniquely determined loop. From this path and loop paths
  labelled by $(a,\$)$ or $(\$,a)$ branch off.}
\label{fig:unaryautomaton}
\end{figure}

Let $D$ be a multiple of the lengths of the loops in $\fsa{A}$
that also exceeds the number of states in $\fsa{A}$.

Let $\fsa{A}$ have initial state $q_0$ and transition function
$\delta$. Consider a word $uvw \in \conv L(\fsa{A})$, where $v =
b^\beta$ for some $b \in \{a,\$\}^{2}$ and $\beta \geq D$. Suppose
that $(q_0,u)\delta = q$. When $\fsa{A}$ is in state $q$ and reads
$v$, it completes a loop before finishing reading $v$. So $v$
factorizes as $v'v''v'''$, with $|v''| > 0$, such that $(q,v')\delta =
(q,v'v'')\delta = q'$. Assume that $|v'|$ is minimal, so that $q'$ is
the first state on the loop that $\fsa{A}$ encounters while reading
$v$. Assume further that $|v''|$ is minimal, so that $\fsa{A}$ makes
exactly one circuit around the loop while reading $v''$. Now, by
definition, $D$ is a multiple of $|v''|$. Let $m = D/|v''|$. So
$|v(v'')^{m+1}v'''| = |v| + D$. By the pumping lemma,
$uv'(v'')^{m+1}v'''w \in \conv L(\fsa{A})$.

Consider what this means in terms of the pair $\vec{p} =
(a^{p_1},a^{p_2})$ such that $\conv(\vec{p}) = uvw$. Since $v' \in
b^*$, it follows that $uv'(v'')^{m+1}v'''w = \conv(a^{p_1+q_1},a^{p_{2}+q_{2}})$, where
\[
q_j = \begin{cases} 0 & \text{if $p_j \leq |u|$} \\
D & \text{if $p_j \geq |uv|$}.
\end{cases}
\]
(Note that either $p_j \leq |u|$ or $p_j \geq |uv|$ since $v \in b^*$
for a fixed pair $b \in \{a,\$\}^{2}$.) Therefore we have the following:

\begin{pumprule}
\label{rule:up}
If the components of a pair in $\Lambda(R,\phi)$ can be
partitioned into those that are of length less than $l \in
\nset$ and those that have length at least $l + D$, then [the word
encoding] this pair can be pumped so as to increase by $D$ the
lengths of those components that are at least $l + D$ letters
long and yield another [word encoding a] pair in $\Lambda(R,\phi)$.
\end{pumprule}

(Notice that this also applies when both components have length
at least $D$; in this case, set $l = 0$.) 

With the same setup as above, suppose $|v| \geq 2D$. Then $\fsa{A}_i$
must follow the loop labelled by $v''$ starting at $q'$ at least
$m = D/|v''|$ times. That is, $v$ factorizes as
$v'(v'')^m\tilde{v}'''$. By the pumping lemma, $uv'\tilde{v}''' \in
\conv L(\fsa{A}_i)$ and $|v'\tilde{v}'''| = |v| - D$. Therefore, we also
have the following:

\begin{pumprule}
\label{rule:down}
If the components of a pair in $\Lambda(R,\phi)$ can be divided
into those that are of length less than $l \in \nset$ and those
that have length at least $l + 2D$, then [the word encoding] this
pair can be pumped so as to decrease by $D$ the length of those
components that are at least $l + 2D$ letters long and yield
another [word encoding a] pair in $\Lambda(R,\phi)$.
\end{pumprule}

This ability to pump so as to increase or decrease lengths of
components of a pair by a constant $D$ lends itself to a very useful
diagrammatic representation of the unary FA-pre\-sent\-a\-tion
$(a^*,\phi)$. Consider a grid of $D$ rows and infinitely many
columns. The rows, from bottom to top, are $B[0],\ldots,B[D-1]$. The
columns, starting from the left, are $C[0],C[1],\ldots$. The point in
column $C[x]$ and row $B[y]$ corresponds to the word $a^{xD+y}$. For
example, in the following diagram, the distinguished point is in
column $C[3]$ and row $B[2]$ and so corresponds to $a^{3D+2}$:

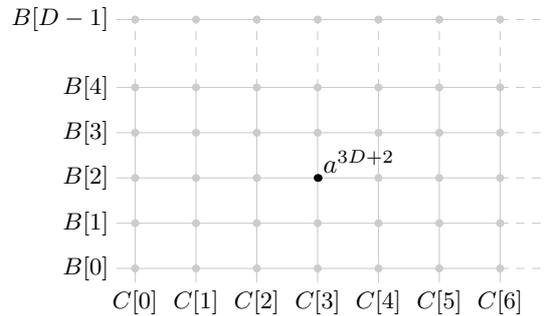
\inpargraphic{\jobname-pumpingdiagram1.eps}

The power of such diagrams is due to a natural correspondence between
pumping as in \fullref{Pumping rules}{rule:up} and~\ref{rule:down} and certain
simple manipulations of pairs of points in the diagram. Before describing this
correspondence, we must set up some notation. We will not distinguish
between a point in the grid and the word to which it corresponds. The
columns are ordered in the obvious way, with $C[x] < C[x']$ if and
only if $x < x'$.

For any element $u \in a^*$, let $b(u)$ be the index of the row
containing $u$ and let $c(u)$ be the index of the column containing
$u$. For brevity, write $B[u]$ for $B[b(u)]$ and $C[u]$ for $C[c(u)]$.
Extend the notation for intervals on $\nset$ to intervals of
contiguous columns. For example, for $x,x' \in \nset$ with $x \leq
x'$, let $C[x,x')$ denotes the set of elements in columns
  $C[x],\ldots,C[x'-1]$, and $C(x,\infty)$ denotes the set of
  elements in columns $C[x+1],C[x+2],\ldots$.

Define for every $n \in \zset$ a partial map $\tau_n
: a^* \to a^*$, where $a^k\tau_n$ is defined to be $a^{k+nD}$ if $k+nD
\geq 0$ and is otherwise undefined. Notice that if $n \geq 0$, the map
$\tau_n$ is defined everywhere. In terms of the diagram, $a^k\tau_n$
is the element obtained by shifting $a^k$ to the right by $n$ columns
if $n \geq 0$ and to the left by $-n$ columns if $n < 0$. The values
of $k$ and $n < 0$ for which $a^k\tau_n$ are undefined are precisely
those where shifting $a^k$ to the left by $-n$ columns would carry it
beyond the left-hand edge of the diagram.

Now, elements of $\Lambda(R,\phi)$ are pairs of
words, and thus can be viewed as arrows in the diagram.

Consider a pair $(p,q)$ in $\Lambda(R,\phi)$, viewed as an arrow in
the diagram. If the arrow $(p,q)$ neither starts nor ends in $C[0]$,
so that both $p$ and $q$ are of length at least $D$ and so the word
encoding $(p,q)$ can be pumped before both components in accordance
with \fullref{Pumping rule}{rule:up}. This corresponds to shifting
both components rightwards by one column. This rightward shifting of
components can be iterated arbitrarily many times to yield new arrows.

Similarly, if the arrow $(p,q)$ starts and ends in non-adjacent
columns, then $p$ and $q$ differ in length by at least $D$ and hence
the word encoding $(p,q)$ can be pumped between these two components
in accordance with \fullref{Pumping rule}{rule:up}. This corresponds
to shifting the rightmost of $p$ or $q$ rightwards by one column.
This rightward shifting of arrows can be iterated arbitrarily many
times to yield new arrows.

Hence we have the following diagrammatic version of \fullref{Pumping
  rule}{rule:up}:

\begin{arrowrule}
\label{arrowrule:up}
Consider an element $(p,q)$ of $\Lambda(R,\phi)$, viewed as an arrow
in the diagram.
\begin{enumerate}

\item If $(p,q)$ neither starts nor ends in $C[0]$, then for any $k
  \in \nset$ the arrow $(p\tau_k,q\tau_k)$ obtained by shifting $(p,q)$ right
  by $k$ columns also lies in $\Lambda(R,\phi)$.

\item If $c(p) > c(q) + 1$, then for any $k \in \nset$ the arrow
  $(p\tau_k,q)$ obtained by shifting $p$ right by $k$ columns also
  lies in $\Lambda(R,\phi)$.

\item If $c(q) > c(p) + 1$, then for any $k \in \nset$ the arrow
  $(p,q\tau_k)$ obtained by shifting $q$ right by $k$ columns also
  lies in $\Lambda(R,\phi)$.

\end{enumerate}

\end{arrowrule}

On the other hand, if the arrow $(p,q)$ neither starts nor ends in
$C[0,1]$, then $p$ and $q$ are both of length at least $2D$ and so the
word encoding $(p,q)$ can be pumped before both components in
accordance with \fullref{Pumping rule}{rule:down}. This corresponds to
shifting both components leftwards by one column. This leftward
shifting can be iterated to yield new arrows for as long as neither
$p$ nor $q$ lies in $C[0,1]$.

Similarly, if $c(p)$ and $c(q)$ differ by at least $2$, then $p$ and
$q$ differ in length by at least $2D$ and hence the word encoding
$(p,q)$ can be pumped between these two components in accordance with
\fullref{Pumping rule}{rule:down}. This corresponds to shifting the
rightmost of $p$ or $q$ leftwards by one column.  This leftward shift
of one end of the arrow can be iterated to yield new arrows for as
long as there are at least two columns between $p$ and $q$.

Hence we have the following diagrammatic version of \fullref{Pumping
  rule}{rule:down}:

\begin{arrowrule}
\label{arrowrule:down}
Consider an element $(p,q)$ of $\Lambda(R,\phi)$ viewed as an arrow
in the diagram.
\begin{enumerate}

\item If $(p,q) \in C[h,\infty)$, then for any $k \in \nset$ with
  $0 < k < h$ the arrow $(p\tau_{-k},q\tau_{-k})$ obtained by shifting
  $(p,q)$ left by $k$ columns also lies in $\Lambda(R,\phi)$.

\item If $c(p) > c(q) + h$, then for any $k \in \nset$ with
  $0 < k < h$ the arrow $(p\tau_{-k},q)$ obtained by shifting
  $p$ left by $k$ columns also lies in $\Lambda(R,\phi)$.

\item If $c(q) > c(p) + h$, then for any $k \in \nset$ with
  $0 < k < h$ the arrow $(p,q\tau_{-k})$ obtained by shifting
  $q$ left by $k$ columns also lies in $\Lambda(R,\phi)$.

\end{enumerate}
\end{arrowrule}

An arrow from $p$ to $q$, where $|c(p)-c(q)| \leq 1$ is called a
\defterm{short arrow}. Any other arrow is called a \defterm{long
  arrow}. In both \fullref{Arrow rules}{arrowrule:up} and
\ref{arrowrule:down}, case~1 applies to both long and short arrows;
cases~2 and~3 apply only to long arrows. Hence, as shown in
\fullref{Figure}{fig:arrowrules}, the two ends of a short arrow
maintain the same relative position when \fullref{Arrow
  rules}{arrowrule:up} and \ref{arrowrule:down} are applied; those of
a long arrow need not.

\begin{figure}[tb]
\centerline{\includegraphics{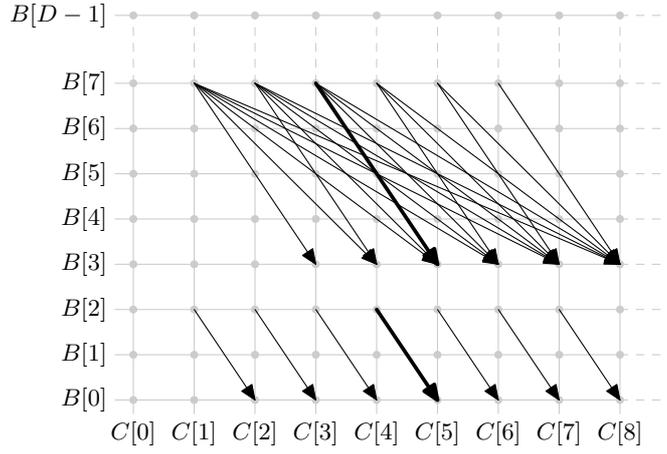}}
\caption{Arrows arising from a long arrow and a short arrow via
  \fullref{Arrow rules}{arrowrule:up} and \ref{arrowrule:down}.}
\label{fig:arrowrules}
\end{figure}

\section{Transitive closure}

Consider an FA-pre\-sent\-a\-tion $(L,\phi)$ for a relational
structure $\mathcal{S}$ that includes a binary relation $R$. Let $R_1$
and $R_2$ be, respectively, the reflexive and symmetric closures of
$R$:
\begin{align*}
(x,y) \in R_1 &\iff ((x,y) \in R) \lor (x = y); \\
(x,y) \in R_2 &\iff ((x,y) \in R) \lor ((y,x) \in R).
\end{align*}
Then since $R_1$ and $R_2$ are defined by first-order formulae,
\fullref{Proposition}{prop:firstorderreg} shows that
$\conv\Lambda(R_1,\phi)$ and $\conv\Lambda(R_2,\phi)$ are regular.

In contrast, the transitive closure $R^+$ is not defined by a
first-order formula and so \fullref{Proposition}{prop:firstorderreg}
does not apply. Indeed, since there exist FA-pre\-sent\-a\-ble directed
graphs where reachability is undecidable \cite[Examples~2.4(iv)
  \&~3.17]{rubin_survey}, $\conv\Lambda(R^+,\phi)$ is not regular
in the case of general FA-pre\-sent\-a\-tions.

However, this section is dedicated to showing that, in the particular
case of unary FA-pre\-sent\-a\-ble structures,
$\conv\Lambda(R^+,\phi)$ \emph{i}s
regular. \fullref{Theorem}{thm:reftransclosure} below proves that
$\conv\Lambda(R^*,\phi)$ is regular, where $R^*$ is the reflexive and
transitive closure of $R$. The corresponding result for $R^+$ can be
proved by a similar method. Notice that
\fullref{Theorem}{thm:reftransclosure} considerably strengthens both
\cite[Lemma~7.10]{blumensath_diploma}, which essentially states (in
different language) that the reflexive and transitive closure of a
unary FA-presentable \emph{unary function} is unary FA-presentable,
and \cite[Corollary~7.6]{khoussainov_unarygraphalgorithmic}, which
states that in a unary FA-presentable undirected graph of finite
degree, the reachability relation is regular.

\begin{theorem}
\label{thm:reftransclosure}
Let $\mathcal{S}$ be a structure admitting an injective unary FA-pre\-sent\-a\-tion
$(a^*,\phi)$. Let $R$ be some binary relation in the signature of
$\mathcal{S}$. Then $\Lambda(R^*,\phi)$ is regular, where $R^*$
denotes the reflexive and transitive closure of $R$. Hence $\mathcal{S}$ augmented
by $R^*$ is also unary FA-pre\-sent\-a\-ble.
\end{theorem}

\begin{proof}
Suppose the diagram for $(a^*,\phi)$ has $D$ rows. The relation $R$ is
binary, so elements of $\Lambda(R,\phi)$ may be viewed as arrows
between points in the diagram. Two points will then lie in
$\Lambda(R^*,\phi)$ if and only if they are linked by a directed path
(possibly of length zero). Thus, throughout the
proof, we will reason mainly about arrows and directed paths.

Informally, the overall strategy is to break up a path from $p$ to $q$
into three parts: a subpath from $p$ to the minimum column the path
visits, a subpath from one vertex of this column to another, and a
subpath from this column to $q$, and then to replace these subpaths
with subpaths that are either short or that can be broken into
segments between columns at most $2D+2$ apart in a way that makes the
replacement subpath recognizable by an automaton. This constant $2D+2$
becomes vital only later, in a pumping argument in the proof of
\fullref{Lemma}{lem:sthreesfour} below, but it is introduced
immediately because certain other constants are defined in terms of
it.

For every $p,q \in C[0,2D+2]$ such that there is a directed path
from $p$ to $q$, fix some such path $\alpha_{p,q}$. Let
$k_1$ be the maximum of the lengths of the various paths
$\alpha_{p,q}$.

\begin{lemma}
\label{lem:nearcolumnpaths}
There is a constant $k_2$ with the following property: for all points
$p$ and $q$ with $|c(p)- c(q)| \leq 2D+2$, if there is a directed path
from $p$ to $q$ that does not visit the column $C[0]$, then there is a
directed path from $p$ to $q$ of length at most $k_2$ that does not visit
$C[0]$.
\end{lemma}

\begin{proof}
For convenience, let $h = 2D+2$. Consider all pairs of distinct points
$p'$, $q'$ in $C[1,h+1]$. For each such pair, consider whether
there is some $i \in \nset^0$ such that there is a directed
path from $p'\tau_i$ to $q'\tau_i$ that does not visit
$C[0]$. If some such $i$ exists, let $i(p',q')$ be the minimum such
$i$ and let $\alpha_{p',q'}$ be a directed path from
$p\tau_{i(p',q')}$ to $q\tau_{i(p',q')}$. Let $k_2$ be the maximum of
the lengths of the various paths $\alpha_{p',q'}$.

\begin{figure}[tb]
\centerline{\includegraphics{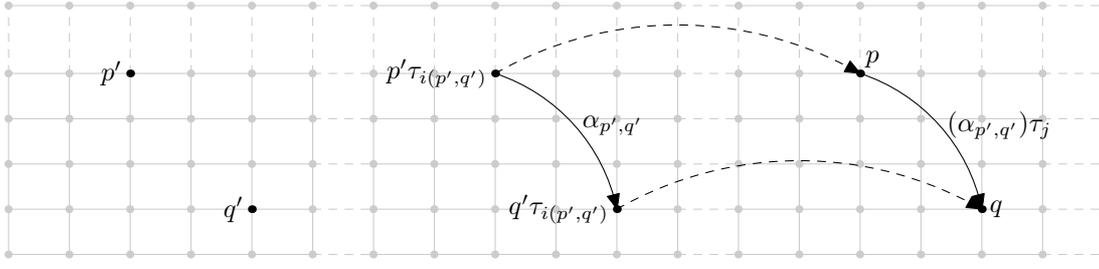}}
\caption{A path of length at most $k_2$ from $p$ to $q$ is obtained by
  shifting the path $\alpha_{p',q'}$ to the right by $j = c(p) -1
  -i(p',q')$ columns. Notice that the path $\alpha_{p',q'}$ may visit
  vertices to the left of $p\tau_{i(p',q')}$ or the right of
  $q\tau_{i(p',q')}$: the only thing that is guaranteed is that its
  length is at most $k_2$ and that it does not visit any point of
  $C[0]$.}
\label{fig:nearcolumnpaths}
\end{figure}

Now let $p$ and $q$ be arbitrary with $c(p) \leq c(q) \leq c(p) +
h$. The case where $c(q) \leq c(p) \leq c(q)+h$ is similar. Suppose
there is a directed path from $p$ to $q$ that does not visit
$C[0]$. Let $p' = p\tau_{-(c(p)-1)}$ and $q' =
q\tau_{-(c(p)-1)}$. Then $p'$ and $q'$ lie in $C[1,h+1]$. Since
there is a directed path from $p$ to $q$, the quantity
$i(p',q')$ and the path $\alpha_{p',q'}$ are defined. Furthermore, the
minimality of $i(p',q')$ ensures that $i(p',q') \leq c(p)-1$. Let $j =
c(p)-1-i(p',q')$. Then the directed path $(\alpha_{p',q'})\tau_{j}$ is
defined by \fullref{Arrow rule}{arrowrule:up} and runs from $p$ to
$q$. (See \fullref{Figure}{fig:nearcolumnpaths}.)  So there is a
directed path from $p$ to $q$ of length at most $k_2$ that does not visit $C[0]$.
\end{proof}

Let $k = \max\{k_1,k_2\}+1$. Let
\[
S = R^{\leq k} = E \cup R \cup (R \circ R) \cup (R \circ R \circ R) \cup \ldots \cup (\underbrace{R \circ \cdots \circ R}_{\text{$k$ times}}),
\]
where $E$ is the equality relation. Then $S$ is first-order definable
and so $\Lambda(S,\phi)$ is regular.

Define the relation
\begin{align}
V &= \bigl\{(t_0\phi,t_n\phi) : (\exists n \in \nset)(\exists t_1,\ldots,t_{n-1} \in a^*)(\forall i \in \{0,\ldots,n-1\}) \label{eq:transclosure1} \\
 &\qquad\qquad\qquad\qquad\qquad\bigl((|t_i| < |t_{i+1}| \land (t_i\phi,t_{i+1}\phi) \in S\bigr)\bigr\},\nonumber
\end{align}
and dually
\begin{align}
V' &= \bigl\{(t_0\phi,t_n\phi) : (\exists n \in \nset)(\exists t_1,\ldots,t_{n-1} \in a^*)(\forall i \in \{0,\ldots,n-1\}) \label{eq:transclosure2} \\
 &\qquad\qquad\qquad\qquad\qquad\bigl((|t_{i+1}| < |t_i| \land (t_i\phi,t_{i+1}\phi) \in S\bigr)\bigr\}.\nonumber
\end{align}
Notice that $V \subseteq R^+$ and $V' \subseteq R^+$, and that
\eqref{eq:transclosure1} and \eqref{eq:transclosure2} are not
first-order definitions because in the quantifications $(\exists
t_1,\ldots,t_{n-1})$, the number of variables is arbitrary. Notice
further that it follows immediately from their definitions that $V$ and $V'$ are
transitive.

[The definitions of $V$ and $V'$ could be formulated using quantification over
  the domain of $\mathcal{S}$, but \eqref{eq:transclosure1} and
  \eqref{eq:transclosure2} are notationally more useful since we will
  reason using the words $t_i \in a^*$.]

\begin{lemma}
\label{lem:s3regular}
The relations $\Lambda(V,\phi)$ and  $\Lambda(V',\phi)$ are regular.
\end{lemma}

\begin{proof}
We prove the regularity of $\Lambda(V,\phi)$; the argument for
$\Lambda(V',\phi)$ is symmetric.

Let $\fsa{A}$ be an automaton recognizing
$\conv\Lambda(S,\phi)$. Construct an automaton $\fsa{B}$ recognizing
$\conv\Lambda(V,\phi)$ as follows.

The automaton $\fsa{B}$ reads words of the form $\conv(a^m,a^{m+n})$
for $m \in \nset^0$, $n \in \nset$; that is, where the word on the
left-hand track is shorter than the one on the right-hand track. So
the automaton need only read input symbols from $\{(a,a),(\$,a)\}$.

The operation of the automaton $\fsa{B}$ consists in running two
copies $\fsa{A}_1$, $\fsa{A}_2$ of the automaton
$\fsa{A}$ simultaneously. The copy $\fsa{A}_1$ faithfully
simulates the operation of $\fsa{A}$. The second copy
$\fsa{A}_2$ always follows the $(a,a)$ transition, regardless of
whether the input is $(a,a)$ or $(\$,a)$. Furthermore, whenever
$\fsa{A}_1$ is in a state corresponding to an accept state $f$ of
$\fsa{A}$, and $\fsa{A}_2$ is in a state corresponding to
any state $q$ of $\fsa{A}$, the automaton $\fsa{B}$ can make
an $\emptyword$-transition so that both $\fsa{A}_1$ and
$\fsa{A}_2$ are in states corresponding to $q$. The accept
states of $\fsa{B}$ are those where $\fsa{A}_1$ is in a
state corresponding to an accept state of $\fsa{A}$, and
$\fsa{A}_2$ is in any state.

Prove that $L(\fsa{B}) \subseteq \conv\Lambda(V,\phi)$ as follows:
Suppose $\fsa{B}$ reads $h_0$ symbols $(a,a)$ before switching to
reading symbols $(\$,a)$. Suppose it makes $n-1$ of the
$\emptyword$-transitions described above, after having read a total of
$h_1$, $h_2$, \ldots, $h_{n-1}$ symbols $(\$,a)$, and that it accepts
after having read a total of $h_n$ symbols.

Now, if $\fsa{B}$ makes an $\emptyword$-transition so that $\fsa{A}_1$
and $\fsa{A}_2$ are both in state $q$, then if it makes another
$\emptyword$-transition without reading any input, then it does not
change its state, for both $\fsa{A}_1$ and $\fsa{A}_2$ continue to
have state $q$. Similarly, if $\fsa{B}$ makes an
$\emptyword$-transition while reading symbols $(a,a)$, it does not
change its state. Therefore $\fsa{B}$ can read the same input by not
making any $\emptyword$-transitions while reading symbols $(a,a)$, and
not making more than one consecutive $\emptyword$-transition. Hence we
can assume without loss that $h_0 < h_1 < \ldots < h_n$. For each $i
\in \{0,\ldots,n\}$, set $t_i = a^{h_i}$.

After having read the word $(a,a)^{h_0}(\$,a)^{h_1-h_0}$, the automaton
$\fsa{B}$ makes an $\emptyword$-transition. So, by the construction of
$\fsa{B}$, the first simulated copy of $\fsa{A}$ must have been in an
accept state. Hence $(a,a)^{h_0}(\$,a)^{h_1-h_0} \in L(\fsa{A})$ and so
$(t_0\phi,t_1\phi) \in S$.

Immediately after making the $\emptyword$-transition after having read
a total of $h_i$ symbols (where $i \in \{1,\ldots,n-1\}$), the first
simulated copy of $\fsa{A}$ is in the state $\fsa{A}$ enters after
having read $(a,a)^{h_i}$. The automaton $\fsa{B}$ reads
$(\$,a)^{h_{i+1}-h_i}$ and then makes another $\emptyword$-transition
or accepts; in either case the first simulated copy of $\fsa{A}$ is in
an accept state. Thus $\fsa{A}$ must accept
$(a,a)^{h_i}(\$,a)^{h_{i+1}-h_i} \in L(\fsa{A})$ and so
$(t_i\phi,t_{i+1}\phi) \in S$.

Thus $(t_i\phi,t_{i+1}\phi) \in S$ and $|t_i| < |t_{i+1}|$ for
all $i \in \{0,\ldots,n-1\}$. So $(t_0\phi,t_n\phi) \in V$.

To see that $\conv\Lambda(V,\phi) \subseteq L(\fsa{B})$,
proceed as follows. Let $(s_0\phi,s_n\phi) \in V$, and let
$t_1,\ldots,t_{n-1}$ be such that $(t_i\phi,t_{i+1}\phi) \in S$ and
$|t_i| < |t_{i+1}|$. Let $h_i = |t_i|$. Then
$\conv(a^{h_i},a^{h_{i+1}}) \in L(\fsa{A})$. Then $\fsa{B}$ can read
$(a,a)^{h_0}(\$,a)^{h_1-h_0}$, taking the first simulated copy of
$\fsa{A}$ to an accept state, then make an
$\emptyword$-transition. After the $i$-th $\emptyword$-transition, the
first simulated copy of $\fsa{A}$ is in the state $\fsa{A}$ would be
in after reading $(a,a)^{h_i}$. So reading $(\$,a)^{h_{i+1}-h_i}$
brings this copy of $\fsa{A}$ into an accept state, where $\fsa{B}$
can either make another transition or accept. Hence $\fsa{B}$ accepts
after reading $(a,a)^{h_0}(\$,a)^{h_n-h_0}$. So
$\conv\Lambda(V,\phi) \subseteq L(\fsa{B})$.
\end{proof}

Define $T$ to consist of all pairs $(p\phi,q\phi)$ such that there is
a directed path of non-zero length from $p$ to $q$ in which all
visited points except the first lie in $C(p,\infty)$. (This
entails $c(q) > c(p)$.)

Define $U$ to consist of all pairs $(p\phi,q\phi)$ such that there is
a directed path from $p$ to $q$ in which all visited points except the
last lie in $C(q,\infty)$. (This entails $c(p) > c(q)$.)

\begin{lemma}
\label{lem:sthreesfour}
$T \subseteq V$ and $U \subseteq V'$.
\end{lemma}

\begin{proof}
We prove that $T \subseteq V$. A symmetric argument proves $U \subseteq V'$.

We are required to show that $(p\phi,q\phi) \in T \implies
(p\phi,q\phi) \in V$. We will prove this statement via induction on
$c(q) - c(p)$.

First, if $c(q) - c(p) \leq 2D+2$, then there is a directed path from $p$ to
$q$ of length at most $k$ (by the definition of $k$ in terms of $k_1$ and $k_2$). Hence
$(p\phi,q\phi) \in S$. Since $|p| < |q|$, it follows from \eqref{eq:transclosure1} that
$(p\phi,q\phi) \in V$.

So suppose $c(q) - c(p) = h > 2D+2$ and that the result holds for all
$p',q'$ with $c(q') - c(p') < h$. Let $\alpha$ be a directed path from $p$ to
$q$ in which all visited points except the first lie in
$C(p,\infty)$, as per the definition of $T$. There are several
cases to consider:

\begin{enumerate}

\item The path $\alpha$ visits some point in $C(p,q)$. Let $x$ be
  the last such point on $\alpha$ in this range. Let $\beta$ be the
  subpath of $\alpha$ up to and including this visit to $x$, and let
  $\gamma$ be the subpath from $x$ to $q$. Then $\beta$ is a directed
  path that shows that $(p\phi,x\phi) \in T$, and $\gamma$ is a
  directed path that shows that $(x\phi,q\phi) \in T$, since the
  choice of $x$ ensures that $\gamma$ never visits $C[0,x]$
  again. Hence, since $c(x) - c(p)$ and $c(q) - c(x)$ are both less
  than $h$, the induction hypothesis applies to show that
  $(p\phi,x\phi) \in V$ and $(x\phi,q\phi) \in V$. Since $V$ is
  transitive, it follows that $(p\phi,q\phi) \in V$.

\item The path $\alpha$ does not visit any point in
  $C(p,q)$. Then the first arrow in $\gamma$ is a long arrow that runs from $p$ to
  some point $x \in C[q,\infty)$. Now consider two sub-cases:

\begin{enumerate}

\item $c(x) - c(q) \leq 2D+2$. Then there is a directed path from $x$
  to $q$ of length at most $k_2$. Hence there is a path from $p$ to
  $q$ of length at most $k$ (since $k \geq k_2+1$). Hence
  $(p\phi,q\phi) \in S$. Since $|p| < |q|$, it follows from
  \eqref{eq:transclosure1} that $(p\phi,q\phi) \in V$.

\item $c(x) - c(q) > 2D+2$. The subpath $\beta$ from $x$ to $q$ has to
  visit or bypass the $2D+2$ columns in $C(q,x)$. One of two
  cases must hold: either the subpath $\beta$ includes some arrow
  between two points $s$ and $t$ with $c(t) - c(s) > 2$, or the subpath
  $\beta$ visits at least $D+1$ columns between $C[x]$ and
  $C[q]$. Consider each of these sub-sub-cases in turn; in both we will
  construct a new directed path $\alpha'$ from $p$ to $q$:

\begin{enumerate}

\item The subpath $\beta$ includes some arrow between points $s$ and
  $t$ with $c(t) - c(s) > 2$. Let $\gamma$ be the subpath of $\beta$
  from $x$ to $s$ and let $\delta$ be the subpath of $\beta$ from $t$
  to $q$. Since the path $\beta$ does not visit any point in
  $C[0,q)$ and $c(q) > 2D+2$, the path $(\gamma)\tau_{-1}$ exists
    by \fullref{Arrow rule}{arrowrule:down}. Let $\alpha'$ be the path
    formed by concatenating the arrow from $p$ to $x\tau_{-1}$ (which
    exists by \fullref{Arrow rule}{arrowrule:down}), the path
    $(\gamma)\tau_{-1}$, the arrow from $s\tau_{-1}$ to $t$ (which
    exists by \fullref{Arrow rules}{arrowrule:up} and
    \ref{arrowrule:down}), and the path $\delta$.

\item The subpath $\beta$ visits at least $D+1$ columns between $C[x]$
  and $C[q]$. Therefore, since there are only $D$ distinct rows in the
  diagram, the pigeonhole principle shows that the subpath $\beta$ visits
  two points $s$ and $t$ (in that order) with $c(x) \geq c(s) > c(t)
  \geq c(q)$ and $b(s) = b(t)$; furthermore, we can choose $s$ and $t$
  with $c(s)-c(t) \leq 2D+2$. Let $\gamma$ be the subpath of $\beta$
  from $x$ to $s$ and let $\delta$ be the subpath of $\beta$ from $t$
  to $x$. Since the path $\beta$ does not visit any point in
  $C[0,q)$ and $c(q) > 2D+2$, the path
    $(\gamma)\tau_{-(c(s)-c(t))}$ exists by \fullref{Arrow
      rule}{arrowrule:down}. Let $\alpha'$ by the path formed by
    concatenating the arrow from $p$ to $x\tau_{-(c(s)-c(t))}$ (which
    exists by \fullref{Arrow rule}{arrowrule:down}), the path
    $(\gamma)\tau_{-(c(s)-c(t))}$, and the path $\delta$.

\end{enumerate}
In either sub-sub-case, the path $\alpha'$ also runs from $p$ to $q$,
but the first arrow of $\alpha'$ ends at a point $x'$ where $c(x) >
c(x') \geq c(x') - 2D-2$. So replacing $\alpha$ by $\alpha'$ and
iterating this process eventually yields a path $\alpha$ from $p$ to
$q$ where $c(x) - c(q) \leq 2D+2$, reducing this sub-case to
sub-case~(a).\qedhere
\end{enumerate}
\end{enumerate}
\end{proof}

Let
\[
W = S \cup (S \circ V) \cup (V' \circ S) \cup (V' \circ S \circ V).
\]
Notice that $\Lambda(W,\phi)$ is regular and that $W \subseteq R^*$.

\begin{lemma}
\label{lem:trans}
Suppose there is a directed path from $p$ to $q$. Then
$(p\phi,q\phi) \in W$.
\end{lemma}

\begin{proof}
Let $\alpha$ be a directed path from $p$ to $q$. Let $i \in \nset^0$
be maximal such that all points on $\alpha$ lie in $C[i,\infty)$. Let
  $x$ be the point in $C[i]$ that $\alpha$ visits first, and let
  $\beta$ be the subpath of $\alpha$ from $p$ to $x$. Let $y$ be the
  point in $C[i]$ that $\alpha$ visits last, and let $\gamma$ be the
  subpath of $\alpha$ from $y$ to $q$.

Then there is path from $x$ to $y$ of length at most $k$ (by the
definition of $k$ in terms of $k_1$ and $k_2$), and so $(x\phi,y\phi)
\in S$. If $p \neq x$, then $\beta$ is a path of non-zero length from
$p$ to $x$, every point of which, except the last, lies in
$C(i,\infty)$ by the choice of $x$; hence $(p\phi,x\phi) \in
U$. If $y \neq q$, then $\gamma$ is a path of non-zero length from $y$
to $q$, every point of which, except the first, lies in
$C(i,\infty)$; hence $(y\phi,q\phi) \in T$.

Therefore, there are four cases:
\begin{enumerate}

\item $p = x$ and $y = q$. Then $(p\phi,q\phi) = (x\phi,y\phi) \in S$;
\item $p \neq x$ and $y = q$. Then $(p\phi,q\phi) = (p\phi,y\phi) = (p\phi,x\phi) \circ (x\phi,y\phi) \in U \circ S$;
\item $p = x$ and $y \neq q$. Then $(p\phi,q\phi) = (x\phi,q\phi) = (x\phi,y\phi) \circ (y\phi,q\phi) \in S \circ T$;
\item $p \neq x$ and $y \neq q$. Then $(p\phi,q\phi) = (p\phi,x\phi) \circ (x\phi,y\phi) \circ (y\phi,q\phi) \in U \circ S \circ T$.

\end{enumerate}
Since $T \subseteq V$ and $U \subseteq V'$ by
\fullref{Lemma}{lem:sthreesfour} and its dual, in each case
$(p\phi,q\phi) \in W$.
\end{proof}

We can now complete the proof of
\fullref{Theorem}{thm:reftransclosure}. As noted above, $W \subseteq
R^*$. By \fullref{Lemma}{lem:trans}, $R^* \subseteq W$. Therefore $W =
R^*$, and hence $\Lambda(R^*,\phi) = \Lambda(W,\phi)$ is regular.
\end{proof}

As noted above, the regularity of $\Lambda(R^+,\phi)$ can be proved by
a similar method. The proof is slightly more technical, since one has
to exclude paths of length zero, but conceptually the same. We state
the result in full for completeness:

\begin{theorem}
\label{thm:transclosure}
Let $\mathcal{S}$ be a structure admitting an injective unary FA-pre\-sent\-a\-tion
$(a^*,\phi)$. Let $R$ be some binary relation in the signature of
$\mathcal{S}$. Then $\Lambda(R^+,\phi)$ is regular, where $R^+$
denotes the transitive closure of $R$. Hence $\mathcal{S}$ augmented
by $R^+$ is also unary FA-pre\-sent\-a\-ble.
\end{theorem}

\begin{corollary}
\label{corol:equivrelgen}
Let $\mathcal{S}$ be a structure admitting an injective unary
FA-pre\-sent\-a\-tion $(a^*,\phi)$. Let $R$ be some binary relation in
the signature of $\mathcal{S}$. Let $Q$ be the equivalence relation
generated by $R$. Then $\Lambda(Q,\phi)$ is regular. Hence
$\mathcal{S}$ augmented by $Q$ is also unary FA-pre\-sent\-a\-ble.
\end{corollary}

\begin{proof}
Let $R'$ be the symmetric closure of $R$. Then $\Lambda(R',\phi)$ is
regular since $R'$ is first-order definable in terms of $R$. Since the
equivalence relation $Q$ is the reflexive and transitive closure
$(R')^*$ of $R'$, the relation $\Lambda(Q,\phi)$ is regular by
\fullref{Theorem}{thm:reftransclosure}.
\end{proof}

\section{Binary relations}

This section is devoted to characterizing unary FA-pre\-sent\-a\-ble binary
relations (\fullref{Theorem}{thm:binrelchar}), with the aim of
subsequently giving useful characterizations of unary FA presentable
quasi-orders (\fullref{Theorem}{thm:qosetchar}), partial orders
(\fullref{Theorem}{thm:posetchar}), and tournaments
(\fullref{Theorem}{thm:tournamentchar}).

These characterizations have a common form. A structure of one of
these species is unary FA-pre\-sent\-a\-ble if it can be obtained by
extending a finite structure of the same species in a particular
`periodic' fashion that we call `propagation'. The way in which the
finite structure extends is determined by a collection of
distinguished five-element subsets and the relations holding between
these and the rest of the structure. The importance of the subsets
having five elements is to ensure that transitivity is preserved when
an infinite structure is obtained by propagating a finite one (see
the comments following the proof of \fullref{Lemma}{lem:exttransitive}).

\subsection{FA-foundational binary relations}
\label{subsec:foundational}

Let $\rr$ be a binary relation defined on a finite set $Q$. We will
consider $(Q,\rr)$ as a directed graph with vertex set $Q$ and edges
$\rr$, so that there is at most one directed edge from $p$ to $q$ for
any $p,q \in Q$.

Suppose $Q$ is equipped with a distinguished collection of disjoint
subsets $P_0,\ldots,P_{n-1}$ called \defterm{seeds}, that fulfil the
following conditions. Each seed $P_k$ consists of five elements
$p^{(k)}_1$, $p^{(k)}_2$, $p^{(k)}_3$, $p^{(k)}_4$, $p^{(k)}_5$ such
that the following conditions are satisfied for $k,l \in
\{0,\ldots,n-1\}$ (including the possibility that $k = l$) and $q \in
Q' = Q - (P_0 \cup \ldots \cup P_{n-1})$:
\begin{enumerate}

\begin{figure}[p]
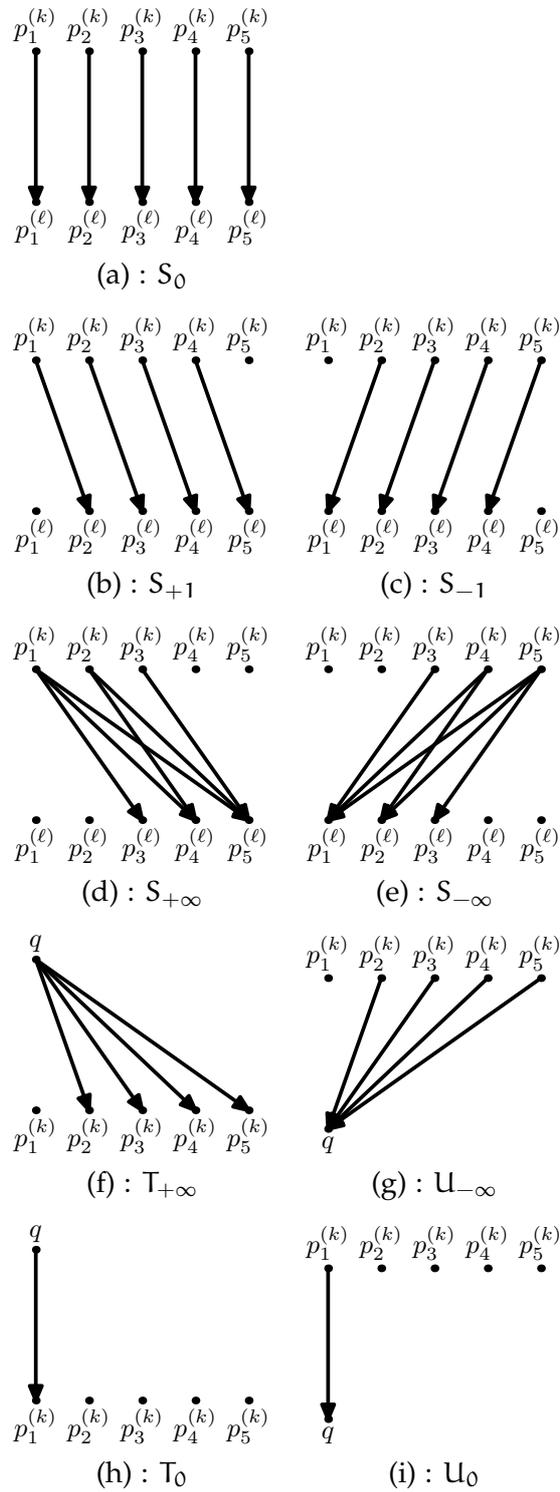

\centerline{%
\begin{tabular}{cc}
\includegraphics{\jobname-qoset-base-weak-same.eps} \\
(a) : $S_0$ \\[3mm]
\includegraphics{\jobname-qoset-base-weak-right.eps} &
\includegraphics{\jobname-qoset-base-weak-left.eps} \\
(b) : $S_{+1}$ & (c) : $S_{-1}$ \\[3mm]
\includegraphics{\jobname-qoset-base-strong-right.eps} &
\includegraphics{\jobname-qoset-base-strong-left.eps} \\
(d) : $S_{+\infty}$ & (e) : $S_{-\infty}$ \\[3mm]
\includegraphics{\jobname-qoset-base-strong-singletonout.eps} &
\includegraphics{\jobname-qoset-base-strong-singletonin.eps} \\
(f) : $T_{+\infty}$ & (g) : $U_{-\infty}$ \\[3mm]
\includegraphics{\jobname-qoset-base-weak-singletonout.eps} &
\includegraphics{\jobname-qoset-base-weak-singletonin.eps} \\
(h) : $T_0$ & (i) : $U_0$
\end{tabular}%
}
\caption{Conditions on edges between $P_k$, $P_l$, and $q$, where $k,l \in \{0,\ldots,n-1\}$ and $q \in Q'$.}
\label{fig:baseconnections}
\end{figure}

\item $p^{(k)}_1 \rr p^{(l)}_1 \iff p^{(k)}_2 \rr p^{(l)}_2 \iff
  p^{(k)}_3 \rr p^{(l)}_3 \iff p^{(k)}_4 \rr p^{(l)}_4 \iff p^{(k)}_5
  \rr p^{(l)}_5$. That is, either all or none of the edges in
  \fullref{Figure}{fig:baseconnections}(a) run from $P_k$ to
  $P_l$. (There may be other edges between $P_k$ and $P_l$ that are
  not shown in the figure, but either all the edges shown here are present
  or none are. The same caveat applies to the remaining conditions.) If
  all of these edges are present, we say there is an $S_0$ connection
  from $P_k$ to $P_l$.

\item $p^{(k)}_1 \rr p^{(l)}_2 \iff p^{(k)}_2 \rr p^{(l)}_3 \iff
  p^{(k)}_3 \rr p^{(l)}_4 \iff p^{(k)}_4 \rr p^{(l)}_5$. That is,
  either all or none of the edges in
  \fullref{Figure}{fig:baseconnections}(b) run from $P_k$ to $P_l$. If
  all these are present, we say there is an $S_{+1}$ connection from
  $P_k$ to $P_l$.

\item $p^{(k)}_2 \rr p^{(l)}_1 \iff p^{(k)}_3 \rr p^{(l)}_2 \iff
  p^{(k)}_4 \rr p^{(l)}_3 \iff p^{(k)}_5 \rr p^{(l)}_4$. That is,
  either all or none of the edges in
  \fullref{Figure}{fig:baseconnections}(c) run from $P_k$ to $P_l$. If
  all these edges are present, we say there is an $S_{-1}$ connection
  from $P_k$ to $P_l$.

\item $p^{(k)}_1 \rr p^{(l)}_3 \iff p^{(k)}_1 \rr p^{(l)}_4 \iff
  p^{(k)}_1 \rr p^{(l)}_5 \iff p^{(k)}_2 \rr p^{(l)}_4 \iff p^{(k)}_2
  \rr p^{(l)}_5 \iff p^{(k)}_3 \rr p^{(l)}_5$. That is, either all or
  none of the edges in \fullref{Figure}{fig:baseconnections}(d) run
  from $P_k$ to $P_l$. If all these edges are present, we say there is
  an $S_{+\infty}$ connection from $P_k$ to $P_l$.

\item $p^{(k)}_3 \rr p^{(l)}_1 \iff p^{(k)}_4 \rr p^{(l)}_1 \iff
  p^{(k)}_5 \rr p^{(l)}_1 \iff p^{(k)}_4 \rr p^{(l)}_2 \iff p^{(k)}_5
  \rr p^{(l)}_2 \iff p^{(k)}_5 \rr p^{(l)}_3$. That is, either all or
  none of the edges in \fullref{Figure}{fig:baseconnections}(e) run
  from $P_k$ to $P_l$. If all these edges are present, we say there is
  an $S_{-\infty}$ from $P_k$ to $P_l$.

\item $q \rr p^{(k)}_2 \iff q \rr p^{(k)}_3 \iff q \rr p^{(k)}_4 \iff
  q \rr p^{(k)}_5$. That is, either all or none of the edges in
  \fullref{Figure}{fig:baseconnections}(f) run from $q$ to $P_k$. If
  all these edges are present, we say there is a $T_{+\infty}$ from
  $q$ to $P_k$.

\item $p^{(k)}_2 \rr q \iff p^{(k)}_3 \rr q \iff p^{(k)}_4 \rr q \iff
  p^{(k)}_5 \rr q$. That is, either all or none of the edges in
  \fullref{Figure}{fig:baseconnections}(g) run from $P_k$ to $q$. If
  all these edges are present, we say there is a $U_{-\infty}$
  connection from $P_k$ to $q$.

\end{enumerate}
A finite binary relation equipped with such a collection of
distinguished subsets is called a \defterm{unary FA-foundational
  binary relation}.

For convenience, define the following additional connections for $q
\in Q'$ and $k \in \{0,\ldots,n-1\}$:
\begin{enumerate}
\item $q \rr p^{(k)}_1$. That is, the edge in
  \fullref{Figure}{fig:baseconnections}(h) runs between $q$ and
  $p^{(k)}_1$. If this edge is present, we say there is a $T_0$
  connection from $q$ to $P_k$.

\item $p^{(k)}_1 \rr q$. That is, the edge in
  \fullref{Figure}{fig:baseconnections}(i) runs between $p^{(k)}_1$
  and $q$. If this edge is present, we say there is a $U_0$ connection
  from $P_k$ to $q$.
\end{enumerate}

Let $k,l \in \{0,\ldots,n-1\}$. Notice that every possible edge from
$P_k$ to $P_l$ is part of exactly one connection $S_{-\infty},
S_{-1}, S_0, S_{+1}, S_{+\infty}$. Thus the set of edges from $P_k$ to
$P_l$ is made up of a union (possibly empty) of these
connections. Similarly, the set of edges from $q \in Q'$ to $P_k$ is
made up of a union of $T_0$ and $T_{+\infty}$, and similarly the set
of edges from $P_k$ to $q$ is made up of a union of $U_0$ and
$U_{-\infty}$.

\subsection{Propagating an FA-foundational binary relation}
\label{subsec:propagating}

Extend $Q$ to an infinite set as follows. For each $k \in
\{0,\ldots,n-1\}$, let $\ext{P}_k = \{p^{(k)}_i : i \in \nset\}$. Let
\[
\ext{Q} = Q' \cup \bigcup_{k=0}^{n-1} \ext{P}_k.
\]
That is, to obtain $\ext{Q}$ from $Q$, each seed $P_k$
of $Q$ is extended to an infinite subset $\ext{P}_k$.

We now describe how to extend $\rr$ to a binary relation $\rre$ on
$\ext{Q}$. Define $\rre$ as follows: first, for $p,q \in Q'$ as
follows: $p \rre q$ if $p \rr q$. For any $k,l \in \{0,\ldots,n-1\}$
and $q \in Q'$:
\begin{enumerate}

\begin{figure}[p]
\centerline{%
\begin{tabular}{cc}
\includegraphics{\jobname-qoset-propagation-weak-same.eps} \\
(a) : $\ext{S}_0$ \\[3mm]
\includegraphics{\jobname-qoset-propagation-weak-right.eps} &
\includegraphics{\jobname-qoset-propagation-weak-left.eps} \\
(b) : $\ext{S}_{+1}$ & (c) : $\ext{S}_{-1}$ \\[3mm]
\includegraphics{\jobname-qoset-propagation-strong-right.eps} &
\includegraphics{\jobname-qoset-propagation-strong-left.eps} \\
(d) : $\ext{S}_{+\infty}$ & (e) $\ext{S}_{-\infty}$ \\[3mm]
\includegraphics{\jobname-qoset-propagation-strong-singletonout.eps} &
\includegraphics{\jobname-qoset-propagation-strong-singletonin.eps} \\
(f) : $\ext{T}_{+\infty}$ & (g) : $\ext{U}_{-\infty}$ \\[3mm]
\includegraphics{\jobname-qoset-propagation-weak-singletonout.eps} &
\includegraphics{\jobname-qoset-propagation-weak-singletonin.eps} \\
(h) : $\ext{T}_0$ & (i) : $\ext{U}_0$
\end{tabular}%
}
\caption{Extension of $\rr$ to $\rre$, where $q \in Q'$ and $k,l \in \{0,\ldots,n-1\}$.}
\label{fig:extensions}
\end{figure}

\item If $p^{(k)}_1 \rr p^{(l)}_1$, then $p^{(k)}_i \rre p^{(l)}_i$
  for all $i \in \nset$. That is, if there is an edge
  $(p^{(k)}_1,p^{(l)}_1) \in \rr$, then all the edges shown in
  \fullref{Figure}{fig:extensions}(a) are present. [We will discuss
    why some edges are shown as bold and some as normal weight later in this
    subsection. There may be other edges between $\ext{P}_k$ and
    $\ext{P}_l$ not shown in this figure. These remarks apply to the
    other cases below.] If all these edges are present, we say there
  is an $\ext{S}_0$ connection from $\ext{P}_k$ to $\ext{P}_l$.

\item If $p^{(k)}_1 \rr p^{(l)}_2$, then $p^{(k)}_i \rre
  p^{(l)}_{i+1}$ for all $i \in \nset$. That is, if there is an edge
  from $(p^{(k)}_1,p^{(l)}_2) \in \rr$, then all the edges shown in
  \fullref{Figure}{fig:extensions}(b) are present. If all these edges
  are present, we say there is an $\ext{S}_{+1}$ connection from $\ext{P}_k$
  to $\ext{P}_l$.

\item If $p^{(k)}_2 \rr p^{(l)}_1$, then $p^{(k)}_{i+1} \rre
  p^{(l)}_i$ for all $i \in \nset$. That is, if there is an edge
  $(p^{(k)}_2,p^{(l)}_1) \in \rr$, then all the edges shown in
  \fullref{Figure}{fig:extensions}(c) are present. If all these edges
  are present, we say there is an $\ext{S}_{-1}$ connection from $\ext{P}_k$
  to $\ext{P}_l$.

\item If $p^{(k)}_1 \rr p^{(l)}_3$, then $p^{(k)}_i \rre
  p^{(l)}_{i+j}$ for all $i,j \in \nset$ with $j \geq 2$. That is, if
  there is an edge $(p^{(k)}_1,p^{(l)}_3) \in \rr$, then all the edges
  shown in \fullref{Figure}{fig:extensions}(d) are present.  (For
  clarity, only edges that both start and end in the scope of this
  diagram are shown. The same applies to the following diagrams.) If
  all these edges are present, we say there is an $\ext{S}_{+\infty}$
  connection from $\ext{P}_k$ to $\ext{P}_l$.

\item If $p^{(k)}_3 \rr p^{(l)}_1$, then $p^{(k)}_{i+j} \rre
  p^{(l)}_i$ for all $i,j \in \nset$ with $j \geq 2$. That is, if
  there is an edge $(p^{(k)}_3,p^{(l)}_1) \in \rr$, then all the edges
  shown in \fullref{Figure}{fig:extensions}(e) exist. If all these
  edges are present, we say there is an $\ext{S}_{-\infty}$ connection
  from $\ext{P}_k$ to $\ext{P}_l$.

\item If $q \rr p^{(k)}_2$, then $q \rre p^{(k)}_i$ for all $i \in
  \nset$ with $i \geq 2$. That is, if there is an edge $(q,p^{(k)}_2)
  \in \rr$, then all the edges shown in
  \fullref{Figure}{fig:extensions}(f) are present. If all these edges
  are present, we say there is an $\ext{T}_{+\infty}$ connection from
  $q$ to $\ext{P}_l$.

\item If $p^{(k)}_2 \rr q$, then $p^{(k)}_i \rre q$ for all $i \in
  \nset$ with $i \geq 2$. That is, if there is an edge $(p^{(k)}_2,q)
  \in \rr$, then all the edges shown in
  \fullref{Figure}{fig:extensions}(g) are present. If all these edges
  are present, we say there is an $\ext{U}_{-\infty}$ connection from
  $\ext{P}_k$ to $q$.

\end{enumerate}

For convenience, define the following additional connections for $q
\in Q'$ and $k \in \{1,\ldots,n\}$:
\begin{enumerate}
\item If $q \rr p^{(k)}_1$, then $q \rre p^{(k)}_1$. That is, the edge
  in \fullref{Figure}{fig:extensions}(h) runs from $q$ to
  $p^{(k)}_1$. If this edge is present, we say that there is a
  $\elt{T}_0$ connection from $q$ to $\ext{P}_k$. Notice that a $T_0$
  connection from $q$ to $P_k$ and a $\ext{T}_0$ connection from $q$
  to $\ext{P}_k$ consist of the same edge.

\item If $p^{(k)}_1 \rr q$, then $p^{(k)}_1 \rre q$. That is, the edge
  in \fullref{Figure}{fig:extensions}(i) runs from $p^{(k)}_1$ to
  $q$. If this edge is present, we say there is a $\elt{U}_0$
  connection from $\ext{P}_k$ to $q$. Notice that a $U_0$ connection
  from $P_k$ to $q$ and a $\ext{U}_0$ connection from $\ext{P}_k$ to
  $q$ consist of the same edge.

\end{enumerate}

Observe that each edge in any diagram in
\fullref{Figure}{fig:baseconnections} also appears in the
corresponding diagram in \fullref{Figure}{fig:extensions}. That is, $p
\rr q$ if and only if $p \rre q$ for any $p,q \in Q$. Thus $\rre$
genuinely extends $\rr$. (The bold edges in
\fullref{Figure}{fig:extensions} are those present in $\rr$.)

Notice further that a $S_\sigma$ connection from $P_k$ to $P_l$
gives rise to an $\ext{S}_\sigma$ connection from $\ext{P}_k$ to
$\ext{P}_l$ for any $\sigma \in
\{-\infty,-1,0,+1,+\infty\}$. Similarly, a $T_\sigma$ connection
from $q \in Q'$ to $P_k$ gives rise to a $\ext{T}_\sigma$
connection from $q$ to $\ext{P}_l$ for $\sigma \in \{0,+\infty\}$,
and a $U_\sigma$ connection from $P_k$ to $q \in Q'$ gives rise to
a $\ext{U}_\sigma$ connection from $q$ to $\ext{P}_l$ for $\sigma
\in \{-\infty,0\}$.

Let $k,l \in \{0,\ldots,n-1\}$. Notice that every possible edge from
$\ext{P}_k$ to $\elt{P}_l$ is part of exactly one connection
$\ext{S}_{-\infty}, \ext{S}_{-1}, \ext{S}_0, \ext{S}_{+1},
\ext{S}_{+\infty}$. Thus the set of edges from $\ext{P}_k$ to $\ext{P}_l$ is
made up of a union (possibly empty) of these connections. Similarly,
the set of edges from $q \in Q'$ to $\ext{P}_k$ is made up of a union of
$\ext{T}_0$ and $\ext{T}_{+\infty}$, and similarly the set of edges from $\ext{P}_k$ to
$q$ is made up of a union of $\ext{U}_0$ and $\ext{U}_{-\infty}$.

\subsection{Characterization of binary relations}

\begin{theorem}
\label{thm:binrelchar}
A binary relation is unary FA-pre\-sent\-a\-ble if and only if it can be
obtained by propagating a unary FA-foundational binary relation.
\end{theorem}

\begin{proof}
\textit{First part.} Let $(\ext{Q},\rre)$ be a binary relation
obtained by propagating a unary FA-foundational binary
relation. Retain notation from
\fullref{Subsections}{subsec:foundational} and
\ref{subsec:propagating}. Notice first that if $(\ext{Q},\rre)$ is
finite (which can happen if the number of seeds $n$ is
$0$), it is unary FA-pre\-sent\-a\-ble. So assume $(\ext{Q},\rre)$ is
infinite, which requires $n > 0$.

Define a representation map $\phi : a^* \to \ext{Q}$ as follows.
Elements of $Q'$ are represented by the words
$a^0,\ldots,a^{|Q'|-1}$. The elements $p^{(k)}_i \in P_k$, where $k \in
\{0,\ldots,n-1\}$ and $i \in \nset$ are represented by words of the form
$a^{|Q'|+ni+k}$. That is, given a word $a^h$ with $h \geq |Q'|$,
the set $P_k$ to which $a^h\phi$ belongs is determined by the
remainder of dividing $h - |Q'|$ by $n$, and the subscript $i$ is
determined by the (integer) quotient of $h - |Q'|$ by
$n$.

An automaton recognizing $\conv\Lambda(\rre,\phi)$ functions as
follows: while reading each of its two tracks, it stores either the length
of the word read up to a maximum length of $|Q'|$, or the length of
the word modulo $n$. Thus, for the input word on each track, the
automaton knows either which element of $Q'$ is represented by the
input word, or which of the various $P_k$ contains the element
represented by the input word. The automaton also stores the
difference in lengths between the two input words, up to a maximum
difference of $\pm (2n+|Q'|)$. In particular, therefore, in the case
when both input words represent elements $p^{(k)}_i$ and
$p^{(l)}_j$, the automaton knows whether $j-i$ is less than or
equal to $-2$, equal to $-1$, $0$, or $1$, or at least $2$. In the
case when one word represents $q \in Q'$ and the other $p^{(k)}_i$,
the automaton knows whether the subscript $i$ is $1$ or at least $2$.

We claim that this bounded amount of stored information is enough for
the automaton to decide whether the two input words represent elements
related by $\rre$. 

First, if both words represent elements of $Q'$, the automaton can
accept if and only if the two elements (which it stores) are related
by $\rre$.

Second, if the element represented by the left-hand input word is $q
\in Q'$ and the other element is $p^{(k)}_i$, the automaton accepts
either if $i=1$ and there is a $\ext{T}_0$ connection between $q$ and
$\ext{P}_k$, or if $i \geq 2$ and there is a $\ext{T}_{+\infty}$
connection between $q$ and $\ext{P}_k$. Recall that the automaton
stores the element $q$, the index $k$, and whether the subscript $i$
is $1$ or at least $2$.

The case where the element represented by the right-hand input word is $q
\in Q'$ and the other element is $p^{(k)}_i$ is similar.

Third, if the element represented by the left-hand input word is
$p^{(k)}_i$ and the other is $p^{(l)}_j$, then the automaton must
accept if and only if
\begin{itemize}
\item $j-i \geq 2$ and there is a $\ext{S}_{+\infty}$ connection between $\ext{P}_k$ and $\ext{P}_l$, or
\item $j-i = 1$ and there is a $\ext{S}_{+1}$ connection between $\ext{P}_k$ and $\ext{P}_l$, or
\item $j-i = 0$ and there is a $\ext{S}_{0}$ connection between $\ext{P}_k$ and $\ext{P}_l$, or
\item $j-i = -1$ and there is a $\ext{S}_{-1}$ connection between $\ext{P}_k$ and $\ext{P}_l$, or
\item $j-i \leq 2$ and there is a $\ext{S}_{-\infty}$ connection between $\ext{P}_k$ and $\ext{P}_l$.
\end{itemize}
Recall that the automaton knows sufficient information about
$j-i$ and knows the indices $k$ and $l$.

\medskip
\noindent\textit{Second part.} Suppose that $(\ext{Q},\rre)$ is a
unary FA-pre\-sent\-a\-ble binary relation. If $\ext{Q}$ is finite, it is
a unary FA-foundational binary relation with $n=0$, as defined in
\fullref{Subsection}{subsec:foundational}.

So assume without loss of generality that $\ext{Q}$ is infinite and
let $(a^*,\phi)$ be an injective unary FA-pre\-sent\-a\-tion for
$(\ext{Q},\leq)$. Suppose the diagram for $(a^,\phi)$ has $D$
rows. Each element of $\Lambda(\rre,\phi)$ is an arrow between two points in the diagram.

Let $Q'$ be those points represented by words in the leftmost column
$C[0]$ of the diagram; that is, $Q' = (C[0])\phi$. For $k =
0,\ldots,D-1$ and $i \in \nset$, let $p^{(k)}_i$ be the element
represented by the unique word in the row $B[k]$ and column $C[i]$;
that is, $p^{(k)}_i = (B[k] \cap C[i])\phi$. Let $\ext{P}_k = (B[k] -
C[0])\phi$ and $P_k = (B[k] \cap C[1,5])\phi$.

As a consequence of \fullref{Arrow rules}{arrowrule:up} and
\ref{arrowrule:down}, for any $\sigma \in \{-\infty,-1,0,+1,+\infty\}$
and $k, l \in \{0,\ldots,D-1\}$, either there is an $\ext{S}_{\sigma}$
connection from from $\ext{P}_k$ to $\ext{P}_l$, or no edge that is
part of an $\ext{S}_{\sigma}$ connection runs from $\ext{P}_k$ to
$\ext{P}_l$. Similarly, for any $q \in Q'$ and $k \in
\{0,\ldots,D-1\}$, either there is a $\ext{T}_{+\infty}$ connection
from $q$ to $\ext{P}_k$, or no edge that is part of an
$\ext{T}_{+\infty}$ connection runs from $q$ to $\ext{P}_k$, and
either there is a $\ext{U}_{-\infty}$ connection from $\ext{P}_k$ to
$q$, or no edge that is part of an $\ext{U}_{-\infty}$ connection runs
from $\ext{P}_k$ to $q$.

Let $Q = Q' \cup \bigcup_{k=0}^{D-1} P_k$, and let $\rr$ be the
restriction of $\rre$ to $Q$. Then $(Q,\rr)$ is a finite binary
relation equipped with distinguished subsets
$P_0,\ldots,P_{D-1}$. Furthermore, the conditions on connections
between the sets $P_k$ hold (as restrictions of the conditions on
connections between the sets $\ext{P}_k$ in $(\ext{Q},\rre)$) by the
observations in the previous paragraph. Hence $(Q,\rr)$ is a unary
FA-foundational binary relation with seeds $P_0,\ldots,P_{D-1}$. It is
easy to see that propagating $(Q,\rr)$ yields $(\ext{Q},\rre)$.
\end{proof}

\subsection{Preservation of properties}

Preparatory to the characterization results in the next section, we
prove that various properties are preserved in passing from $(Q,\rr)$
to $(\ext{Q},\rre)$ and vice versa. The key to several of the proofs
is the following result:

\begin{lemma}
\label{lem:isomorphicpairs}
For any $x,y \in \ext{Q}$, there exist $x',y' \in Q$ such that:
\begin{enumerate}
\item if $x = y$, then $x'=y'$;
\item the map $x \mapsto x'$ and $y \mapsto y'$ is an isomorphism
  between the induced substructures $\{x,y\}$ and $\{x',y'\}$.
\end{enumerate}
\end{lemma}

\begin{proof}
Consider three cases separately, depending on whether none, one, or
both of $x$ and $y$ lie in $Q$:
\begin{enumerate}

\item Suppose both $x$ and $y$ lie in $Q$. Then let $x'=x$ and $y' = y$;
there is nothing to prove.

\item Suppose only one of $x$ and $y$ lies in $Q$; without loss of
  generality, assume $x \in Q'$ and $y \in \ext{Q} - Q'$. Then $y =
  p^{(k)}_i \in \ext{P}_k$ for some $k \in \{0,\ldots,n-1\}$ and $i
  \in \nset$. Let $x' = x$ and $y' = p^{(k)}_2$. Then the following
  are equivalent: (1) $x' \rr y'$; (2) there is an ${T}_{+\infty}$
  connection from $x'$ to $P_k$; (3) there is an $\ext{T}_{+\infty}$
  connection from $x$ to $\ext{P}_k$; (4) $x \rre y$. Similarly, by
  $U_{-\infty}$ and $\ext{U}_{-\infty}$ connections from $P_k$ to $x$ and
  $\ext{P}_k$ to $x$, we see that $y' \rr x'$ if and only if $y \rre
  x$.  Similarly, by considering $S_{0}$ and $\ext{S}_{0}$ connections
  from $P_k$ to $P_k$ and $\ext{P}_k$ to $\ext{P}_k$, we see that $y'
  \rr y'$ if and only if $y \rre y$. Hence in this second case the map
  is an isomorphism of induced substructures.

\item Suppose neither $x$ nor $y$ lies in $Q$. Then $x,y \in \ext{Q} - Q'$
and hence $x = p^{(k)}_i$ and $y = p^{(l)}_j$ for some $k,l \in
\{0,\ldots,n-1\}$ and $i,j \in \nset$. Let $x' = p^{(k)}_3$ and let
\begin{equation}
\label{eq:isomorphicpairs1}
y' = \left\{\begin{array}{ll}
p^{(l)}_1 & \text{if $i \leq j-2$},\\
p^{(l)}_2 & \text{if $i = j-1$},\\
p^{(l)}_3 & \text{if $i = j$},\\
p^{(l)}_4 & \text{if $i = j+1$},\\
p^{(l)}_5 & \text{if $i \geq j+2$}.
\end{array}\right.
\end{equation}
Then the following are equivalent: (1) $x' \rr y'$; (2) there is an
$S_\sigma$ connection from $P_k$ to $P_l$, where
\[ 
\sigma = \left\{\begin{array}{rl}
\infty & \text{if $i \leq j-2$},\\
+1 & \text{if $i = j-1$},\\
0 & \text{if $i = j$},\\
-1 & \text{if $i = j+1$},\\
-\infty & \text{if $i \geq j+2$};
\end{array}\right.
\]
(3) there is an $\ext{S}_\sigma$ connection from $\ext{P}_k$ to
$\ext{P}_l$; (4) $x \rre y$ (by the choice of $y'$ in
\eqref{eq:isomorphicpairs1}). Similarly, $y' \rr x'$ if and only if $y
\rre x$. Finally, considering the presence or absence of an
$\ext{S}_0$ connection from $\ext{P}_k$ to $\ext{P}_k$ shows that $x'
\rr x'$ if and only if $x \rre x$; similarly, $y' \rr y'$ if and only
if $y \rre y$.  Hence in this third case the map is an isomorphism of
induced substructures.
\end{enumerate}
\end{proof}

Since reflexivity, symmetry, and anti-symmetry are defined by axioms
over two variables, the following three lemmata follow easily from
\fullref{Lemma}{lem:isomorphicpairs}.

\begin{lemma}
\label{lem:extreflexive}
The unary FA-foundational binary relation $(Q,\rr)$ is reflexive if and
only if the propagated binary relation $(\ext{Q},\rre)$ is reflexive.
\end{lemma}

\begin{lemma}
\label{lem:extsymmetric}
The unary FA-foundational binary relation $(Q,\rr)$ is symmetric if and
only if the propagated binary relation $(\ext{Q},\rre)$ is symmetric.
\end{lemma}

\begin{lemma}
\label{lem:extantisymmetric}
The unary FA-foundational binary relation $(Q,\rr)$ is anti-symmetric if and
only if the propagated binary relation $(\ext{Q},\rre)$ is anti-symmetric.
\end{lemma}

The next result is the analogue of
\fullref{Lemmata}{lem:extreflexive}--\ref{lem:extantisymmetric} for
transitivity.

\begin{lemma}
\label{lem:exttransitive}
The unary FA-foundational binary relation $(Q,\rr)$ is transitive if and
only if the propagated binary relation $(\ext{Q},\rre)$ is transitive.
\end{lemma}

\begin{proof}
Suppose first that $(\ext{Q},\rre)$ is transitive. Then $(Q,\rr)$ is
transitive since $\rr$ is the restriction of $\rre$ to $Q$.

Now suppose that $(Q,\rr)$ is transitive. Showing that
$(\ext{Q},\rre)$ is transitive requires consideration of many cases
similar to each other. Suppose there is an edge from $p^{(k)}_h$ to
$p^{(l)}_i$ and an edge from $p^{(l)}_i$ to $p^{(m)}_j$. To prove
transitivity, it is necessary to show that there is an edge from
$p^{(k)}_h$ to $p^{(m)}_j$. The different cases arise from the various
possible connections in which these edges lie.

Every case proceeds in the same way: the edge from $p^{(k)}_h$ to
$p^{(l)}_i$ and the edge from $p^{(l)}_i$ to $p^{(m)}_j$ lie in
connections $\ext{S}_\rho$ and $\ext{S}_\sigma$. By the definition of
propagation, there is a $S_\rho$ connection from $P_k$ to $P_l$ and an
$S_\sigma$ connection from $P_l$ to $P_m$. Transitivity in $(Q,\rr)$
then forces one or more connections $S_\tau$ to hold between $P_k$ and
$P_m$, and propagation then requires $\ext{S}_\tau$ to hold between
$\ext{P}_k$ and $\ext{P}_m$, and one of these connections contains an
edge from $p^{(k)}_h$ to $p^{(m)}_j$. The cases involving one or two
edges to or from elements of $Q'$ are similar.

We will limit ourselves to proving one exemplary case in full
detail and summarizing the others. Consider the case where the edge from $p^{(k)}_h$ to
$p^{(l)}_i$ lies in an $\ext{S}_\infty$ connection and the edge from
$p^{(l)}_i$ to $p^{(m)}_j$ lies an $\ext{S}_{-1}$ connection. That is, $h
+ 2 \leq i$ and $i-1 = j$. Hence $j \geq h+1$.

By the definition of propagation, in $(Q,\rr)$ there is an
${S}_\infty$ connection from ${P}_k$ to ${P}_l$ and an ${S}_{-1}$
connection from ${P}_l$ to ${P}_m$. In particular, there are edges
from $p^{(k)}_1$ to $p^{(l)}_3$, from $p^{(l)}_3$ to
$p^{(m)}_2$, from $p^{(k)}_1$ to $p^{(l)}_4$, and from
$p^{(l)}_4$ to $p^{(m)}_3$. Hence, by the transitivity of
$(Q,\rr)$, there are edges from $p^{(k)}_1$ to $p^{(m)}_2$, and from
$p^{(k)}_1$ to $p^{(m)}_3$. Therefore there are $S_{+\infty}$ and
$S_{+1}$ connections from $P_k$ to $P_m$.

Therefore, by the definition of propagation, there are
$\ext{S}_{+\infty}$ and $\ext{S}_{+1}$ connections from
$\ext{P}_k$ to $\ext{P}_m$. Thus there are edges from
$p^{(k)}_h$ to $p^{(m)}_{j}$ for $j = h+1$ and for $j \geq h+2$, and
thus for all $j \geq h+1$. 

Hence the transitivity condition holds when the first edge from
$p^{(k)}_h$ to $p^{(l)}_i$ lies in $\ext{S}_\infty$ connection from
$\ext{P}_k$ to $\ext{P}_l$ and the second lies edge from $p^{(l)}_i$
to $p^{(m)}_j$ in an $\ext{S}_{-1}$ connection from $\ext{P}_l$ to
$\ext{P}_m$.

The various cases are summarized in
\fullref{Table}{tbl:enforcedconnections}. 

We remark on two further exemplary cases: if there is a $T_0$ connection from $q
\in Q'$ to $P_l$ and a $U_0$ connection from $P_l$ to $s \in
Q'$, then transitivity means there is a single edge (a `$Q'$-edge')
from $q$ to $s$. If there is a $T_0$ connection from $q \in Q'$ to
$P_l$ and a $U_{-\infty}$ connection from $P_l$ to $s \in Q'$,
then transitivity does \emph{not} enforce an edge from $q$ to $s$.
\begin{table}[t]
\centerline{\(\displaystyle\begin{array}{c|c|c|c|c|c|c|c|c|c|}
 & S_{-\infty} & S_{-1} & S_0 & S_{+1} & S_{+\infty} & T_0 & T_{+\infty} & U_0 & U_{-\infty} \\
\hline
S_{-\infty} & S_{-\infty} & S_{-\infty} & S_{-\infty} & S_{-\infty,-1} & S_{*} & \text{---} & \text{---} & U_{-\infty} & U_{-\infty} \\
\hline
S_{-1} & S_{-\infty} & S_{-\infty} & S_{-1} & S_0 & S_{+1,+\infty} & \text{---} & \text{---} & U_{-\infty} & U_{-\infty} \\
\hline
S_0 & S_{-\infty} & S_{-1} & S_0 & S_{+1} & S_{+\infty} & \text{---} & \text{---} & U_{-\infty} & U_0 \\
\hline
S_{+1} & S_{-\infty,-1} & S_{0} & S_{+1} & S_{+\infty} & S_{+\infty} & \text{---} & \text{---} & \text{None} & U_{0,-\infty} \\
\hline
S_{+\infty} & S_{*} & S_{+1,+\infty} & S_{+\infty} & S_{+\infty} & S_{+\infty} & \text{---} & \text{---} & \text{None} & U_{0,-\infty} \\
\hline
T_0 & \text{None} & \text{None} & T_0 & T_{+\infty} & T_{+\infty} & \text{---} & \text{---} & \text{$Q'$ edge} & \text{None} \\
\hline
T_{+\infty} & T_{0,+\infty} & T_{0,+\infty} & T_{+\infty} & T_{+\infty} & T_{+\infty} & \text{---} & \text{---} & \text{None} & \text{$Q'$-edge} \\
\hline
U_0 & \text{---} & \text{---} & \text{---} & \text{---} & \text{---} & S_0 & S_{+\infty} & \text{---} & \text{---} \\
\hline
U_{-\infty} & \text{---} & \text{---} & \text{---} & \text{---} & \text{---} & S_{-1,-\infty} & S_{*} & \text{---} & \text{---} \\
\hline
\end{array}\)}
\caption{Connections enforced by the transitivity of $(Q,\rr)$. Row
  labels show the connection between $P_k$ (or an element of $Q'$) and
  $P_l$ (or an element of $Q'$), column labels show the connection
  between $P_l$ (or an element of $Q'$) and $P_m$ (or an element of
  $Q'$). The cell in a particular row and column shows the
  connection(s) between $P_k$ (or an element of $Q'$) and $P_m$ (or an
  element of $Q'$) enforced by transitivity. The notation $S_*$
  abbreviates $S_{-\infty},S_{-1},S_0,S_1,S_{+\infty}$; the notation
  $S_{+1,+\infty}$ abbreviates $S_{+1},S_{+\infty}$; the notation
  $S_{-\infty,-1}$ abbreviates $S_{-\infty},S_{-1}$. A table for the
  connections required for transitivity of $(\ext{Q},\rre)$ can be
  obtained from this table by replacing each $S_\sigma$, $T_\sigma$,
  or $U_\sigma$ by the corresponding $\ext{S}_\sigma$,
  $\ext{T}_\sigma$, or $\ext{U}_\sigma$.}
\label{tbl:enforcedconnections}
\end{table}
\end{proof}

[The proof of \fullref{Lemma}{lem:exttransitive} shows why the sets
  $P_k$ contain five elements: to ensure that if there is an
  $S_{+\infty}$ connection from $P_k$ to $P_l$ and an $S_{+\infty}$
  connection from $P_l$ to $P_m$, then there is an $S_{+\infty}$
  connection from $P_k$ to $P_m$. This requires considering an edge
  from $p^{(k)}_1$ to $p^{(l)}_3$ and an edge from $p^{(l)}_3$ to
  $p^{(m)}_5$ and applying transitivity in $(Q,\rr)$ to get an edge
  from $p^{(k)}_1$ to $p^{(m)}_5$ and hence an $S_{+\infty}$
  connection from $P_k$ to $P_l$.]

\subsection{Orders}

Equipped with the lemmata from the previous subsection, we can now
characterize unary FA-pre\-sent\-a\-ble quasi-orders and partial
orders. Recall that a \defterm{quasi-order} is a binary relation that is
reflexive and transitive. The following characterization follows
immediately from the lemmata in the previous subsection.

\begin{proposition}
\label{prop:qosetchar}
A quasi-order is unary FA-pre\-sent\-a\-ble if and only if it can be
obtained by propagating a unary FA-foundational quasi-order.
\end{proposition}

\begin{proof}
This is immediate from \fullref{Theorem}{thm:binrelchar} and
\fullref{Lemmata}{lem:extreflexive} and \ref{lem:exttransitive}.
\end{proof}

However, we can improve \fullref{Proposition}{prop:qosetchar} to the
following result:

\begin{theorem}
\label{thm:qosetchar}
A quasi-order is unary FA-pre\-sent\-a\-ble if and only if it can be
obtained by propagating a unary FA-foundational quasi-order in which
every seed $P_k$ is either an anti-chain (with none of
the $p^{(k)}_i$ being comparable), an ascending chain (with $p^{(k)}_1
< p^{(k)}_2 < p^{(k)}_3 < p^{(k)}_4 < p^{(k)}_5$), a descending chain
(with $p^{(k)}_1 > p^{(k)}_2 > p^{(k)}_3 > p^{(k)}_4 > p^{(k)}_5$), or
a strongly connected component (with $p^{(k)}_i \leq p^{(k)}_j$ for all
$i,j \in \{1,\ldots,5\}$).
\end{theorem}

\begin{proof}
Notice that in one direction the result has already been proven:
propagating such a unary FA-foundational quasi-order yields a unary
FA-pre\-sent\-a\-ble quasi-order by
\fullref{Proposition}{prop:qosetchar}.

Therefore let $(\ext{Q},\leq)$ be a unary FA-pre\-sent\-a\-ble quasi-order
and let $(a^*,\phi)$ be an injective unary FA-pre\-sent\-a\-tion. Follow the
second part of the proof of \fullref{Theorem}{thm:binrelchar} to
obtain a unary FA-foundational quasi-order $(Q,\leq)$, with
distinguished sets $P_k$, that, when propagated, yields
$(\ext{Q},\leq)$.  Consider some distinguished set $P_k$ and the
corresponding $\ext{P}_k$.

Suppose there is an $\ext{S}_{+\infty}$ connection from $\ext{P}_k$ to
itself. Then there is an edge from $p^{(k)}_1$ to $p^{(k)}_3$. Let
$a^l,a^m \in a^*$ be such that $a^l\phi = p^{(k)}_1$ and $a^m\phi =
p^{(k)}_3$. By the definition of $\ext{P}_k$ (in the proof of
\fullref{Theorem}{thm:binrelchar}), $b(a^l) = b(a^m) = k$ and $c(a^l)
= 1$ and $c(a^m) = 3$. Thus $m = l+2D$. That is, $(a^l,a^{l+2D}) \in
\Lambda(\leq,\phi)$. By \fullref{Pumping rule}{rule:down},
$(a^l,a^{l+D}) \in \Lambda(\leq,\phi)$. Notice that $b(a^{l+D}) = b(a^l)$ and $c(a^{l+D}) =
c(a^l) + 1 = 2$. Thus there is an edge from $a^l\phi = p^{(k)}_1$ to
$a^{l+D}\phi = p^{(k)}_2$. Hence there is an $\ext{S}_{+1}$ connection
from $\ext{P}_k$ to itself.

Similarly, one can show that if there is an $\ext{S}_{-\infty}$ from
$\ext{P}_k$ to itself, then there is an $\ext{S}_{-1}$ connection from
$\ext{P}_k$ to itself.

If there is an $\ext{S}_{+1}$ connection from $\ext{P}_k$ to itself,
there is an $\ext{S}_{+\infty}$ connection from $\ext{P}_k$ to itself
as a consequence of transitivity. Similarly, if there is an
$\ext{S}_{-1}$ connection from $\ext{P}_k$ to itself, there is an
$\ext{S}_{-\infty}$ connection from $\ext{P}_k$ to itself as a
consequence of transitivity.

Thus from $\ext{P}_k$ to itself, either there are both $\ext{S}_{+1}$
and $\ext{S}_{+\infty}$ connections or there are neither, and
similarly for $\ext{S}_{-1}$ and $\ext{S}_{-\infty}$ connections. Thus
there are four possibilities:
\begin{enumerate}

\item No connections
  $\ext{S}_{-\infty},\ext{S}_{-1},\ext{S}_{+1},\ext{S}_{+\infty}$ from
  $\ext{P}_k$ to itself. Then $\ext{P}_k$ and thus $P_k$ are
  antichains.

\item Connections $\ext{S}_{-\infty}$ and $\ext{S}_{-1}$ from
  $\ext{P}_k$ to itself, but neither $\ext{S}_{+1}$ nor
  $\ext{S}_{+\infty}$. Then $\ext{P}_k$ and thus $P_k$ are descending
  chains.

\item Connections $\ext{S}_{+\infty}$ and $\ext{S}_{+1}$ from
  $\ext{P}_k$ to itself, but neither $\ext{S}_{-1}$ nor
  $\ext{S}_{-\infty}$. Then $\ext{P}_k$ and thus $P_k$ are ascending
  chains.

\item All connections
  $\ext{S}_{-\infty},\ext{S}_{-1},\ext{S}_{+1},\ext{S}_{+\infty}$ from
  $\ext{P}_k$ to itself. Then $\ext{P}_k$ and thus $P_k$ are strongly
  connected components.

\end{enumerate}
\end{proof}

The preceding result yields the following decomposition result for
unary FA-presentable quasi-orders:

\begin{corollary}
\label{corol:qosetsareunions}
Every unary FA-presentable quasi-order decomposes as a finite disjoint
union of trivial quasi-orders, countably infinite ascending chains,
countably infinite descending chains, countably infinite anti-chains,
and countably infinite strongly connected components.
\end{corollary}

We can now characterize unary FA-presentable partial orders:

\begin{theorem}
\label{thm:posetchar}
A partial order is unary FA-pre\-sent\-a\-ble if and only if it can be
obtained by propagating a unary FA-foundational partial order in which
every distinguished set $P_k$ is either an anti-chain (with none of
the $p^{(k)}_i$ being comparable), an ascending chain (with $p^{(k)}_1
< p^{(k)}_2 < p^{(k)}_3 < p^{(k)}_4 < p^{(k)}_5$), or a descending chain
(with $p^{(k)}_1 > p^{(k)}_2 > p^{(k)}_3 > p^{(k)}_4 > p^{(k)}_5$).
\end{theorem}

\begin{proof}
This is immediate from \fullref{Theorem}{thm:qosetchar},
\fullref{Lemma}{lem:extantisymmetric}, and the observation that no
partial order contains a strongly connected component.
\end{proof}

We also have a decomposition result for unary FA-presentable partial
orders, analogous to \fullref{Corollary}{corol:qosetsareunions}:

\begin{corollary}
\label{corol:posetsareunions}
Every unary FA-presentable partial order decomposes as a finite
disjoint union of trivial partial orders, countably infinite ascending
chains, countably infinite descending chains, and countably infinite
anti-chains.
\end{corollary}


\subsection{Tournaments}
\label{subsec:tournaments}

Recall that $(X,\rightarrow)$ (where $\rightarrow$ is a binary
relation on $X$) is a \defterm{tournament} if (when viewed as a
directed graph) every pair of distinct vertices is connected by a
single directed edge, and there is no edge from a vertex to
itself. That is, for every $x,y \in X$ with $x \neq y$, either $x
\rightarrow y$ or $y \rightarrow x$ (but not both), and $x
\not\rightarrow x$ for every $x \in X$. Since this is an axiom over
two variables, the following characterization of unary FA-presentable
tournaments is an easy consequence of
\fullref{Lemma}{lem:isomorphicpairs}.

\begin{theorem}
\label{thm:tournamentchar}
A tournament is unary FA-pre\-sent\-a\-ble if and only if it can be
obtained by propagating a unary FA-foundational tournament.
\end{theorem}

In order to give decomposition result in the spirit of
\fullref{Corollaries}{corol:qosetsareunions} and
\ref{corol:posetsareunions}, we need some terminology.

\begin{definition}
A countably infinite tournament $(X,\rightarrow)$, where $X = \{x_i : i \in \nset\}$, is said to be:
\begin{enumerate}
\item \defterm{complete ascending} if $x_i \rightarrow x_j$ for all $i < j$;
\item \defterm{complete descending} if $x_i \leftarrow x_j$ for all $i < j$;
\item \defterm{near-complete ascending} if $x_i \leftarrow x_{i+1}$
  for all $i$, and also $x_i \rightarrow x_j$ for all $i < j - 1$;
\item \defterm{near-complete descending} if $x_i \rightarrow x_{i+1}$
  for all $i$, and also $x_i \leftarrow x_j$ for all $i < j - 1$.
\end{enumerate}
\end{definition}

\begin{corollary}
\label{corol:tournamentsareunions}
Every unary FA-presentable tournament decomposes as a finite disjoint
union of trivial tournaments, countably infinite complete ascending
tournaments, countably infinite complete descending tournaments,
countably infinite near-ascending tournaments, and countably infinite
near-descending tournaments.
\end{corollary}

\begin{proof}
By \fullref{Theorem}{thm:tournamentchar}, any unary FA-presentable
tournament $(X,\rightarrow)$ is obtained by propagating a unary
FA-foundational tournament. Then $(X,\rightarrow)$ is the finite
disjoint union of the finite set $Q'$ and the various $\ext{P}_k$. For
any $k$, consider the connections that can run from $\ext{P}_k$ to
$\ext{P}_k$. Clearly the presence of an $\ext{S}_0$ connection is
incompatible with $(X,\rightarrow)$ being a tournament. Again from
$(X,\rightarrow)$ being a tournament, we see that there is either an
$\ext{S}_{-\infty}$ connection or an $\ext{S}_{+\infty}$ connection
from $\ext{P}_k$ to $\ext{P}_k$ (but not both). Similarly, there is
either an $\ext{S}_{-1}$ connection or an $\ext{S}_{+1}$ connection
from $\ext{P}_k$ to $\ext{P}_k$.

There are therefore four cases to consider, depending on which
connections $\ext{S}_{-\infty}$ or $\ext{S}_{+\infty}$ and
$\ext{S}_{-1}$ or $\ext{S}_{+1}$ are present:
\begin{itemize}

\item Suppose an $\ext{S}_{-\infty}$ connection and an
  $\ext{S}_{-1}$ connection from $\ext{P}_k$ to $\ext{P}_k$ are present. Then [the
  substructure induced by] $\ext{P}_k$ is a countably infinite
  complete descending tournament.

\item Suppose an $\ext{S}_{+\infty}$ connection and an
  $\ext{S}_{+1}$ connection from $\ext{P}_k$ to $\ext{P}_k$ are present. Then
  $\ext{P}_k$ is a countably infinite complete ascending tournament.

\item Suppose an $\ext{S}_{-\infty}$ connection and an
  $\ext{S}_{+1}$ connection from $\ext{P}_k$ to $\ext{P}_k$ are present. Then
  $\ext{P}_k$ is a countably infinite near-complete descending
  tournament.

\item Suppose an $\ext{S}_{+\infty}$ connection and an
  $\ext{S}_{-1}$ connection from $\ext{P}_k$ to $\ext{P}_k$ are present. Then
  $\ext{P}_k$ is a countably infinite near-complete ascending
  tournament.

\end{itemize}
This completes the proof.
\end{proof}

\section{Trees \& Forests}
\label{sec:trees}

This section is devoted to characterizing unary FA-pre\-sent\-a\-ble
directed and undirected trees and forests. For our purposes, a
directed tree is simply a directed graph that can be obtained by
taking a tree and assigning a direction to each edge.

The characterization results describe unary FA-pre\-sent\-a\-ble trees
as those that can be obtained, via a construction we call
\defterm{attachment}, from finite trees and from two species of
infinite trees that we will define shortly: shallow stars and periodic
paths.

\begin{definition}
\label{def:attachment}
Let $(G,\gamma)$ and $(T,\eta)$ be directed graphs, and let $g \in G$
and $t \in T$ be distinguished vertices. The result of
\defterm{attaching} $(T,\eta)$ at $t$ to the vertex $g$ of
$(G,\gamma)$ is the graph obtained by taking the disjoint union of the
graphs $(G,\gamma)$ and $(T,\eta)$ and identifying the vertices $g$
and $t$.
\end{definition}

We now introduce the two species of infinite trees used in the
characterization results.

\begin{definition}
\label{def:shallowstarperiodicpath}
First we define a \defterm{template}, which comprises a quadruple
$(T,\eta,t_0,t_1)$, where $(T,\eta)$ is a finite directed tree with
vertex set $T$, edge set $\eta$, and $t_0,t_1 \in T$ are distinguished
vertices with $t_0$ being a leaf vertex.

Let $(T,\eta,t_0,t_1)$ be a template. Consider the graph
$\mathcal{S}(T,\eta,t_0,t_1)$ obtained by taking the disjoint union of
countably many copies $(T^{(j)},\eta^{(j)},t_0^{(j)},t_1^{(j)})$ and
amalgamating vertices related by the equivalence relation generated by
\begin{equation}
\label{eq:shallowstarperiodicpathrel}
\mu = \{(t_1^{(j)}, t_0^{(j+1)}) : j \in \nset^0\}.
\end{equation}
In the case where $t_0$ and $t_1$ are the same (leaf) vertex, the
graph $\mathcal{S}(T,\eta,t_0,t_1)$ is the tree obtained by
amalgamating all the leaf vertices $t_0^{(j)} = t_1^{(j)}$ into a
single vertex. In this case, we call the resulting graph a
\defterm{shallow star}, the amalgamated vertex is called the
\defterm{centre} of $\mathcal{S}(T,\eta,t_0,t_1)$, and an edge
incident on the centre is called a \defterm{ray} of
$\mathcal{S}(T,\eta,t_0,t_1)$. Notice that the centre is the unique
vertex of infinite degree in $\mathcal{S}(T,\eta,t_0,t_1)$. Notice
that either all the rays start at the centre, in which case
$\mathcal{S}(T,\eta,t)$ is said to be \defterm{outward}, or all the
rays end at the centre, in which case $\mathcal{S}(T,\eta,t)$ is said
to be \defterm{inward}. (See the example in
\fullref{Figure}{fig:shallowstars}.)

\begin{figure}[tb]
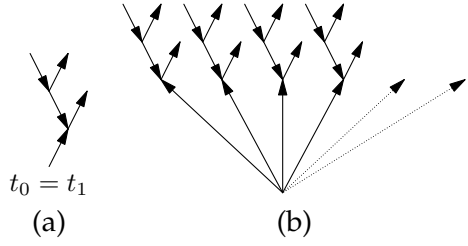

\centerline{%
\begin{tabular}{cc}
\includegraphics{\jobname-tree-shallowstartemplate.eps} & \includegraphics{\jobname-tree-shallowstar.eps} \\
(a) & (b) 
\end{tabular}%
}
\caption{Examples of a template (a) and the corresponding shallow star
  (b).}
\label{fig:shallowstars}
\end{figure}

In the case where $t_0 \neq t_1$, the graph
$\mathcal{S}(T,\eta,t_0,t_1)$ is the tree obtained by amalgamating the
vertex $t_0^{(j+1)}$ in the $(j+1)$-th copy of the template with
$t_1^{(j)}$ in the $j$-th copy. In this case, we call the resulting
graph a \defterm{periodic path}. Notice that there is a unique simple
path $\beta$ in $(T,\eta,t_0,t_1)$ from $t_0$ to $t_1$. Let
$\beta^{(j)}$ be the corresponding path in
$(T^{(j}),\eta^{(j)},t_0^{(j)},t_1^{(j)})$. Then in
$\mathcal{S}(T,\eta,t_0,t_1)$ the concatenation of the paths
$\beta^{(j)}$ form a infinite simple path. This path is called the
\defterm{spine} of $\mathcal{S}(T,\eta,t_0,t_1)$. Notice that the
spine is infinite in only one direction; in the other direction it
begins at $t_0^{(0)}$, which is called the \defterm{base} of
$\mathcal{S}(T,\eta,t_0,t_1)$. (See the example in
\fullref{Figure}{fig:periodicpath}.)

\begin{figure}[tb]
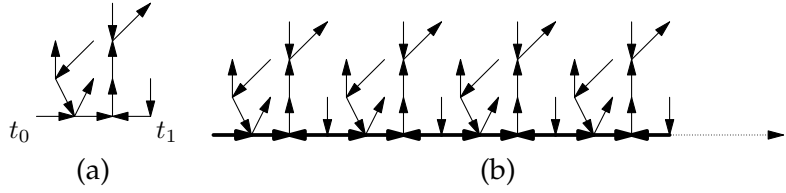

\centerline{%
\begin{tabular}{cc}
\includegraphics{\jobname-tree-periodicpathtemplate.eps} & \includegraphics{\jobname-tree-periodicpath.eps} \\
(a) & (b) 
\end{tabular}%
}
\caption{An example of a template (a) and the corresponding periodic
  path (b), with the spine of the periodic path shown in bold.}
\label{fig:periodicpath}
\end{figure}
\end{definition}

\begin{lemma}
\label{lem:amalgamatingfa}
Let $(G,\gamma)$ be graph admitting an FA-presentation (respectively,
unary FA-presentation) $(L,\phi)$, and let $Q$ be an equivalence
relation on $G$ such that $\Lambda(Q,\phi)$ is regular. Then the graph
$(G_Q,\gamma_Q)$ formed by amalgamating all vertices related by $Q$ is
also FA-presentable (respectively, unary FA-presentable).
\end{lemma}

\begin{proof}
Let $\gamma' = Q \circ \gamma \circ Q$; then $\Lambda(\gamma',\phi)$
is regular and so $(G,\gamma')$ is FA-presentable (respectively, unary
FA-presentable). It is easy to see that $Q$ is a congruence on
$(G,\gamma')$ and that $(G_Q,\gamma_Q)$ is obtained by factoring
$(G,\gamma')$ by $Q$. Hence $(G_Q,\gamma_Q)$ is FA-presentable
(respectively, unary FA-presentable) \cite[Corollary~3.7(iii)]{rubin_survey}.
\end{proof}

\begin{lemma}
\label{lem:attachmentunaryfa}
Let $(G,\gamma)$ and $(T,\eta)$ be unary FA-presentable directed
graphs, and let $g \in G$ and $t \in T$. The graph obtained by
attaching $(T,\eta)$ at $t$ to $g$ is also unary FA-presentable.
\end{lemma}

\begin{proof}
The disjoint union $(S,\sigma)$ of $(G,\gamma)$ and $(T,\eta)$ admits
is unary FA-pre\-sent\-a\-ble by \fullref{Lemma}{lem:union}. Let $R$
be the relation $\{(g,t)\}$ and let $Q$ be the equivalence relation it
generates. Since $R$ is finite, $\Lambda(R,\phi)$ is regular. So, by
\fullref{Corollary}{corol:equivrelgen}, $\Lambda(Q,\phi)$ is
regular. The graph obtained attaching $(T,\eta)$ at $t$ to $g$ is then
unary FA-presentable by \fullref{Lemma}{lem:amalgamatingfa}.
\end{proof}

\begin{lemma}
\label{lem:shallowstarperiodicpathunaryfa}
Shallow stars and periodic paths are is unary FA-pre\-sent\-a\-ble.
\end{lemma}

\begin{proof}
Let $(T,\eta,t_0,t_1)$ be a template. Follow the proof of
\fullref{Lemma}{lem:countableunion} to obtain a unary FA-presentation
$(a^*,\phi)$ for the disjoint union of countably many copies of
$(T,\eta,t_0,t_1)$. Let $p$ and $q$ be such that $a^p\phi$ is the
first copy of the vertex $t_0$ and $a^q\phi$ the first copy of the
vertex $t_1$. Then, by the proof of
\fullref{Lemma}{lem:countableunion},
\[
\Lambda(\mu,\phi) = \{(a^{q+kn},a^{p+(k+1)n}) : k \in \nset\}
\]
(where $\mu$ is as defined in \eqref{eq:shallowstarperiodicpathrel})
and so is regular. Let $Q$ be the equivalence relation generated by
$\mu$; then $\Lambda(Q,\phi)$ is regular by
\fullref{Corollary}{corol:equivrelgen}. Hence
$\mathcal{S}(T,\eta,t_0,t_1)$ is unary FA-presentable by
\fullref{Lemma}{lem:amalgamatingfa}.
\end{proof}

Before stating and proving the characterization theorem for unary
FA-pre\-sent\-a\-ble directed trees, we need the following technical
lemmata:

\begin{lemma}
\label{lem:restrictlongarrows}
Let $(a^*,\phi)$ be an injective unary FA-pre\-sent\-a\-tion for a directed
tree. Then in the diagram, there cannot be a long arrow between two
points in $C[1,\infty)$. Equivalently, any long arrow must either
  start or end in $C[0]$.
\end{lemma}

\begin{proof}

\begin{figure}[tb]
\centerline{\includegraphics{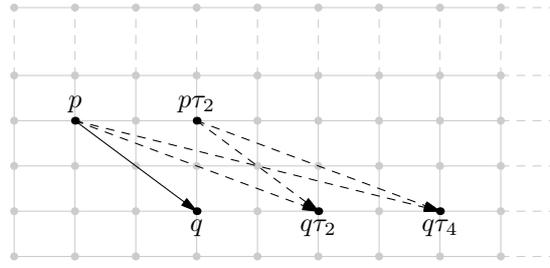}}
\caption{In a unary FA-pre\-sent\-a\-tion for a directed tree, there can be
  no arrow from $p$ to $q$ with $c(q) - c(p) > 1$ and $c(p) \geq
  1$. If such an arrow (solid in the diagram) exists, \fullref{Arrow
    rule}{arrowrule:up} implies the existence of four other arrows
  (dashed in the diagram) that mean the graph cannot be a tree.}
\label{fig:treelongarrow}
\end{figure}

Suppose that $(a^*,\phi)$ is an injective unary FA-pre\-sent\-a\-tion for a
directed tree $T$ and that there is a long arrow from $p$ to $q$ with
$p,q \in C[1,\infty)$. Conside the case when $c(q) > c(p)$; the
  other case is similar. Since $c(q)-c(p) > 1$, \fullref{Arrow
    rule}{arrowrule:up} shows that there are arrows from $p$ to
  $(q)\tau_2$, from $p$ to $(q)\tau_4$, from $(p)\tau_2$ to
  $(q)\tau_4$, and from $(p)\tau_2$ to $(q)\tau_2$, as illustrated in
  \fullref{Figure}{fig:treelongarrow}. Thus in the graph there
  is an undirected cycle $(p)\phi \to ((q)\tau_2)\phi \leftarrow
  ((p)\tau_2)\phi \to ((q)\tau_4)\phi \leftarrow (p)\phi$, which
  contradicts $T$ being a tree.
\end{proof}

\begin{lemma}
\label{lem:shortarrowdegrees}
Let $(a^*,\phi)$ be an injective unary FA-pre\-sent\-a\-tion for a directed
tree. Let $p \in C[1,\infty)$. Then $p\phi$ has degree at most $4D$.
\end{lemma}

\begin{proof}
Suppose $p\phi$ has degree greater than $4D$ and let $k = c(p)$. There
are exactly $4D$ points in $C[0] \cup C[k-1] \cup C[k] \cup C[k+1]$,
so there is some point $q$ in $C[1,k-2]\cup C[k+2,\infty)$
  such that there is an arrow between $p$ and $q$ (in some
  direction). This is a long arrow since $|c(p) - c(q)| \geq 2$, and
  this contradicts \fullref{Lemma}{lem:restrictlongarrows}.
\end{proof}

\begin{lemma}
\label{lem:longarrowdegrees}
Let $(a^*,\phi)$ be an injective unary FA-pre\-sent\-a\-tion for a directed
tree. The following are equivalent:
\begin{enumerate}
\item The diagram contains a long arrow.
\item Some vertex of the graph has infinite degree.
\item Some vertex of the graph has degree greater than $3D$.
\end{enumerate}
\end{lemma}

\begin{proof}
(1 $\implies$ 2) Suppose there is a long arrow between $p$ and $q$ in
  some direction. Assume without loss of generality that $c(q) \geq
  c(p)+2$. Then by \fullref{Arrow rule}{arrowrule:up}, there are
  arrows between $p$ and $q\tau_n$ for all $n$. Hence $p\phi$ has infinite degree.

(2 $\implies$ 3) This is trivial.

(3 $\implies$ 1) Suppose $p\phi$ has degree greater than $3D$. Let $k
  = c(p)$. Then there is some arrow between $p$ and a vertex $q$
  outside $C[k-1,k+1]$ (since this set contains $3D$
  elements). So $|c(p)-c(q)| \geq 2$ and thus the arrow between $p$
  and $q$ is a long arrow.
\end{proof}

\begin{theorem}
\label{thm:directedtreechar}
A directed tree is unary FA-pre\-sent\-a\-ble if and only if it is
isomorphic to a tree obtained by starting from a finite directed tree
and attaching to it finitely many shallow stars (at their centres) and
finitely many periodic paths (at their bases).
\end{theorem}

\begin{proof}
In one direction, the proof is easy. A graph of the prescribed form is
unary FA-pre\-sent\-a\-ble by
\fullref{Lemmata}{lem:attachmentunaryfa},
\ref{lem:shallowstarperiodicpathunaryfa}.

\smallskip\noindent The other direction of the proof is much longer
and more complex. Let $(a^*,\phi)$ be an injective unary
FA-pre\-sent\-a\-tion for a directed tree $(T,\eta)$. Suppose the
diagram for $(a^*,\phi)$ has $D$ rows.

Since $(T,\eta)$ is a tree, the arrows in the diagram also form a
tree. Thus there is a unique simple (undirected) path between any two
points in the diagram.

Let $K$ consist of $C[0] \cup C[1]$ together with all points that lie
on simple paths starting and ending in $C[0] \cup C[1]$.  Then the
subgraph induced by $K$ is a (connected) finite graph. This will be
part of the finite tree in the statement of the theorem.

Suppose there is an arrow between $p$ and $q'$, where $|c(p)-c(q')|
\geq 2$. Then either $p$ or $q'$ lies in $C[0]$ by
\fullref{Lemma}{lem:restrictlongarrows}. Without loss of generality,
suppose $p \in C[0]$. By \fullref{Arrow rules}{arrowrule:down} and
\ref{arrowrule:up}, there is an arrow between $p$ and $q'\tau_n$ for
all $n \geq -c(q')+2$. Let $m \geq -c(q')+2$ be such that $q'\tau_n
\notin K$ for all $n \geq m$. Let $q= q'\tau_m$. Replace $K$ by $K
\cup \{q\tau_l : -c(q) +2 \leq l < 0\}$. Notice that $K$ remains
connected. (The reasoning in this paragraph and the next is
illustrated in \fullref{Figure}{fig:shallowstarfindrays}.)

\begin{figure}[tb]
\centerline{\includegraphics{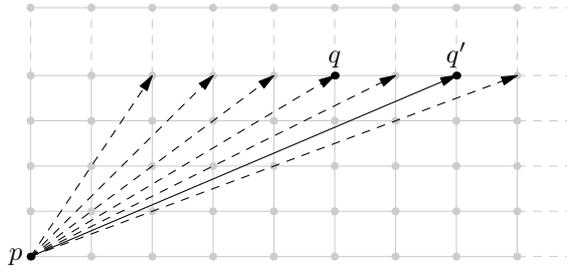}}
\caption{If there is a long arrow (shown as solid in the diagram), it
  must either start or end in $C[0]$; for the sake of illustration, we
  assume it starts at $p \in C[0]$ and ends at $q'$. By \fullref{Arrow
    rules}{arrowrule:down} and \ref{arrowrule:up}, all the arrows
  shown exist. The point $q$ is chosen so that $q\tau_n \notin K$ for
  all $n \geq 0$, and then all points to the left of $q$ are added to
  $K$. (Some of them may lie in $K$ already.) The arrows ending to the
  left of $q$ ensure that $K$ remains connected.}
\label{fig:shallowstarfindrays}
\end{figure}

Then for all $n \geq 0$ the point $q\tau_n$ does not lie in $K$ and
there is an arrow between $p$ and $q\tau_n$. The aim is to show that
the edges corresponding to these arrows are the rays of a shallow star
with centre $p\phi$.

\begin{lemma}
\label{lem:shallowstarpathprops}
For any simple undirected path $\alpha$ starting at
$q\tau_n$, where $n \geq 0$, and not including the arrow from $p$ to $q\tau_n$:
\begin{enumerate}

\item The path $\alpha$ does not visit any point of $K$.

\item For every $m$ with $m \geq -n$, the map $\tau_{m}$ is defined
  for every vertex of $\alpha$, and $(\alpha)\tau_{m}$ is a path in
  the diagram.

\item The path $\alpha$ has length at most $D$.
\end{enumerate}
\end{lemma}

\begin{proof}
\begin{enumerate}

\item Suppose first that $\alpha$ visits some element $x \in
  K$. Without loss of generality, assume $x$ is the first point of $K$
  that $\alpha$ visits. Let $\alpha'$ be the part of $\alpha$ starting
  at $q\tau_n$ and ending at $x$. Since $K$ is a connected set, there
  is a path $\beta$ from $x$ to $p$ wholly within $K$. So the edge
  from $p$ to $q\tau_n$ and the paths $\alpha'$ and $\beta$ form a
  cycle in the diagram, which contradicts $(T,\eta)$ being a tree. So
  the path $\alpha$ does not visit any point of $K$.

\item Suppose that for some $m \geq -n$, the map $\tau_m$ is not
  defined for some point of $\alpha$. (Since $\tau_m$ is always
  defined when $m$ is positive, we know immediately that $m$ is
  negative.) Note that if $\alpha$ included a long arrow, it would
  have to visit $C[0]$, which would contradict part~1 since $C[0]
  \subseteq K$. So $\alpha$ consists only of short arrows. For each
  $t$, let $r_t$ be the $t$-th point visited by $\alpha$. So, since
  $\tau_m$ is undefined for some point on $\alpha$, there is some $s$
  such that $c(r_s) < -m$. Since $\alpha$ consists only of short
  arrows, $c(r_{t+1}) \geq c(r_t) - 1$ for all $t$. Hence there is
  some $s'$ such that $c(r_{s'}) = -m + 1$. Let $\alpha'$ be the
  subpath of $\alpha$ from $q\tau_n$ up to the first point lying in
  $C[r_{s'}]$. Then $\tau_m$ is defined for every point on $\alpha'$
  and the path $(\alpha')\tau_m$ exists by \fullref{Arrow
    rule}{arrowrule:down}, starts at $q\tau_{n+m}$, and ends at some
  point in $C[1]$, which contradicts part~1 applied to the path
  $\alpha'$ since $C[1] \subseteq K$. (This reasoning is illustrated
  \fullref{Figure}{fig:shallowstarpathprops}.)

\begin{figure}[tb]
\centerline{\includegraphics{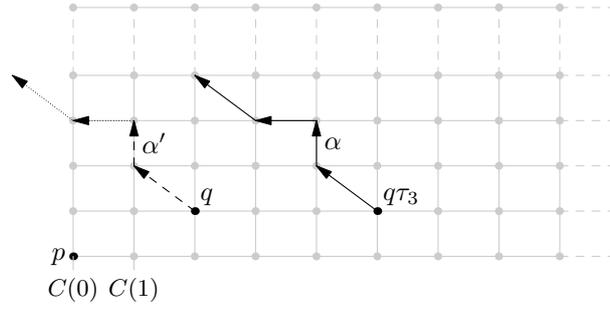}}
\caption{In this illustration, $n = 3$. We suppose there is some path
  $\alpha$ and $m \geq -3$ (here, $m = -3$) such that $(\alpha)\tau_n$
  is undefined. Then $\tau_m$ shifts part of $\alpha$ off the
  left-hand side of the diagram. Since $\alpha$ consists only of short
  arrows, we can choose $\alpha'$ to be that initial part of $\alpha$
  that is shifted by $\tau_n$ to end in $C[1]$.}
\label{fig:shallowstarpathprops}
\end{figure}

 Suppose that $(\alpha)\tau_{m}$ is not a path for some $m \geq
 -n$. \fullref{Arrow rule}{arrowrule:up} shows that $(\alpha)\tau_{m}$
 is a path for all $m \geq 0$. \fullref{Arrow rule}{arrowrule:down}
 shows that $(\alpha)\tau_m$ is a path unless $\tau_m$ shifts some
 point of $\alpha$ to $C[0]$. But in this case, we can choose
 $\alpha'$ as in the previous paragraph so that $(\alpha')\tau_m$ ends
 in $C[1]$ and get the same contradiction. So $(\alpha)\tau_m$ is a
 path in the diagram.

\item Suppose that $\alpha$ has length greater than $D$. Since there
  are only $D$ rows, there are distinct points $x$ and $y$ with $b(x)
  = b(y)$. Interchanging $x$ and $y$ if necessary, assume $c(x) <
  c(y)$. Let $m = c(y) - c(x)$. Then $y = x\tau_m$. So $y$ also lies
  on $\alpha\tau_m$, which is a path in the diagram by part~2. Let
  $\alpha_1$ be the subpath of $\alpha$ from $q\tau_n$ to $y$; let
  $\alpha_2$ be the subpath of $(\alpha)\tau_m$ from $q\tau_{n+m}$ to
  $x\tau_m$. Then the paths $\alpha_1$ and $\alpha_2$ and the arrows
  between $p$ and $q\tau_{n}$ and between $p$ and $q\tau_{n+m}$ form a
  cycle, which contradicts $(T,\eta)$ being a tree. So $\alpha$ cannot
  have length greater than $D$.
\end{enumerate}
\end{proof}

By \fullref{Lemma}{lem:shallowstarpathprops}(2), for any $n \in
\nset^0$, $\beta$ is a simple path starting at $q$ that does not
include the arrow from $p$ to $q$ if and only if $(\beta)\tau_n$ is a
simple path starting at $q\tau_{n}$ that does not include the arrow
from $p$ to $q\tau_{n}$.

Let $F$ consist of all points lying on simple paths starting at $q$
that do not include the arrow from $p$ to $q$. (The set $F$ includes
$q$ itself.) By \fullref{Lemma}{lem:shallowstarpathprops}(3), each of
these paths has length at most $D$. By
\fullref{Lemma}{lem:shallowstarpathprops}(1), none of these paths
visits $K$, and in particular does not visit $C[0]$. Hence, by
\fullref{Lemma}{lem:shortarrowdegrees}, the degree of any point in $F$
is bounded by $4D$ and hence $F$ consists of only finitely many
elements. By \fullref{Lemma}{lem:shallowstarpathprops}(2), every
elements of $F$ lies in $C[2,\infty)$, so $F\tau_n$ is defined
  for all $n \in \nset^0$. By the previous paragraph, all the
  induced subgraphs $F\tau_n$ are isomorphic and so, together with the
  arrows between $p$ to $q\tau_n$, form a shallow star with center
  $p$, and this shallow star is attached at its centre $p\phi$ to some
  vertex of the finite tree.

Since this shallow star contains every point $q\tau_n$, it is clear
that there can be at most $D$ such shallow stars.

Let $L = (a^D)^*F$. Then $L = \bigcup_{n \in \nset^0}
F\tau_n$. The language $L$ consists of all the points corresponding to
vertices of the shallow star centered at $p\phi$ except the point $p$
itself. Since $F$ is finite, $L$ is regular. Since there are at most
$D$ different shallow stars, the language $M'$ formed by the union of
the languages $L$ corresponding to the various shallow stars is
regular. Thus the language $M = a^* - M'$ consisting of words that
either lie outside these shallow stars or are the centres of the
shallow stars, is regular. Hence the subgraph induced by $M\phi$ is
also unary FA-pre\-sent\-a\-ble. Notice that the arrows corresponding
to edges in this induced subgraph are all either short or run between
points in $K$. Thus a bounded number of arrows corresponding to edges
in this induced subgraph start or end at any point of $M$ and thus
every vertex of this induced subgraph has bounded degree. Furthermore,
the original graph is obtained by attaching at most $D$ shallow stars
to this subgraph.

\medskip
Therefore we have reduced to case of a unary FA-pre\-sent\-a\-ble directed
tree $(T,\eta)$ whose vertices all have bounded degree. So we now
assume that $(a^*,\phi)$ is a unary FA-pre\-sent\-a\-tion for such a
graph. By \fullref{Lemma}{lem:shortarrowdegrees}, the diagram for
$(a^*,\phi)$ contains only short arrows. We define $K$ as before, to
consist of $C[0] \cup C[1]$ together with all points that lie on
simple paths starting and ending in $C[0] \cup C[1]$.

Consider any path $\alpha$ starting from some point in $p \in K$ and
otherwise only visiting points outside $K$. Suppose this path visits
points of at least $D+1$ different columns in $C[p,\infty)$. Then
  $\alpha$ visits at least one point in every column in
  $C[p,c(p)+D]$. For each $l \in \{c(p),\ldots,c(p)+D\}$, let
  $x_l$ be the first point $\alpha$ visits in $C[l]$. Since $\alpha$
  consists only of short arrows, it visits some point in each column in
  $C[p,l)$ before it visits $x_l$. Hence $\alpha$ visits $x_k$
    before $x_l$ whenever $k < l$. Since there are only $D$ rows,
    there exist $k$ and $l$ with $k < l$ such that $b(x_k) = b(x_l)$.

\begin{figure}[tb]
\centerline{\includegraphics{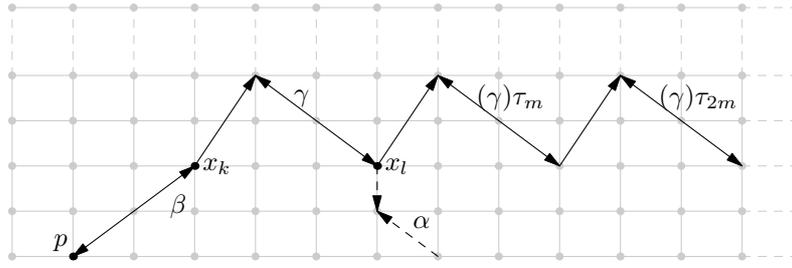}}
\caption{The path $\alpha$ starting at $p$ visits two points $x_k$ and
  $x_l$ lying in the same row. The subpath $\beta$ is the part of
  $\alpha$ from $p$ to $x_k$; the subpath $\gamma$ is the part of
  $\alpha$ from $x_k$ to $x_l$. A small part of $\alpha$ is not
  included in either $\beta$ or $\gamma$. Translating $\gamma$ to the
  right through multiples of $m$ and concatenating the result gives an
  infinite path starting at $p$. We define the infinite peeiodic track
  to be the concatenation of $(\gamma)\tau_m$, $(\gamma)\tau_{2m}$,
  \ldots. The points on $\beta$ and $\gamma$ will be added to $K'$.}
\label{fig:periodicpathfindspine}
\end{figure}

Let $\beta$ be the subpath of $\alpha$ from $p$ to $x_k$ and let
$\gamma$ be the subpath from $x_k$ to $x_l$. Let $m = l - k = c(x_l) -
c(x_k)$. Then $x_l = x_k\tau_m$. Since $\gamma$ is a subpath of
$\alpha$, it lies entirely outside $K$ and so does not visit
$C[0]$. Hence the path $(\gamma)\tau_n$ is defined and present for all
$n \in \nset^0$ by \fullref{Arrow rule}{arrowrule:up}. So the
concatenation of $(\gamma)\tau_m, (\gamma)\tau_{2m},\ldots$ is an
infinite path starting at $x_l$ and formed by `periodic' repetitions
of a translation of $\gamma$. (See
\fullref{Figure}{fig:periodicpathfindspine}.) We will call these paths
the \defterm{infinite periodic tracks}. Notice that this infinite
periodic track does not include $\gamma$ or $\beta$.

The aim is to show that the edges corresponding to the arrows in this
infinite periodic track form the spine of a periodic path. Notice that
there can be at most $D$ distinct infinite periodic tracks, since each
one must visit at least one vertex in each column to the right of its
starting-point (because there are only short arrows in this
diagram). Furthermore, these infinite periodic tracks will be
disjoint.

First of all, we are going to define a finite set $L$ of points. These
points will correspond to the vertices that do not form part of an
infinite periodic path. We need to ensure everything `non-periodic'
lies in $L$. First of all, we will gather all the points lying on the
various paths $\beta$ and $\gamma$ into a set $K'$, then we will deal
with any other points in that might be connected to $K$ by paths that
do not give rise to infinite periodic tracks.

Let $\beta$ and $\gamma$ range over all possible values obtained from
paths $\alpha$ as described above. Then let $K'$ consist of the points
on the various paths $\beta$ and $\gamma$ except the last point of
$\gamma$ (which is also the first point on $(\gamma)\tau_m$). The set
$K'$ is finite since there are at most $D$ infinite periodic tracks
and each path $\beta$ is of bounded length. Notice that all the
infinite periodic tracks start from a point adjacent to a point in
$K'$ (and indeed visit no points of $K'$).

Let $x$ be some point that is connected to $K$ by a path that does not
include any point of an infinite periodic track. Then $x$ is connected
to some point $s \in K$ by a simple path $\alpha$ of bounded length,
since otherwise it would visit at least $D+1$ columns of $C[s,\infty)$
  and so would give rise to an infinite periodic track by the
  reasoning above and thus would include at least the first point (the
  point $x_l$, which is the first point of $\gamma\tau_m$) of this
  path. Therefore all such points $x$ are connected to $K$ by simple
  paths of bounded length. Since the vertices of $(T,\eta)$ have
  bounded degree by assumption, there can only be finitely many such
  points. Let $K''$ be the set of these points. Let $L = K \cup K'
  \cup K''$, and let $L'$ be $L$ together with the starting points of
  the infinite periodic tracks (each of which is adjacent to some
  point of $K' \subseteq L$. Notice that $L$ and $L'$ are connected
  sets, and that any point not in $L$ is connected to some point in
  $K'\subseteq L$ by a unique simple path that visits at least one
  element of an infinite periodic track.

\begin{lemma}
\label{lem:periodicpathprops}
Let $\delta$ be a path starting at a point on some $(\gamma)\tau_{hm}$
(for some $h \in \nset$) and not including any edge of
$(\gamma)\tau_{hm}$ or $\gamma$. (Here $\gamma$ is as in the
definition of an infinite periodic track.) That is, $\delta$ `branches
off' from the infinite periodic track at some point on
$(\gamma)\tau_{hm}$. Then:
\begin{enumerate}

\item The path $\delta$ does not visit any point of $L$.

\item For every $j$ with $j > -h$, the path $(\delta)\tau_{(h+j)m}$ is
  defined and present in the diagram.

\item The path $\delta$ has length at most $D^2+D$.

\end{enumerate}
\end{lemma}

\begin{proof}
\begin{enumerate}
\item Suppose that $\delta$ visits some point $y$ of $L$. Without loss
  of generality, assume $y$ is the first point of $L$ that $\delta$
  visits. Now, if $\delta$ branches off from the infinite track at
  any point except its first point (the first point of
  $(\gamma)\tau_m$), then the second point on $\delta$ must lie
  outside $L$. (Since $L$, by definition, contains no points connected
  to $K \subseteq L$ by a path including a point of an infinite
  track.) On the other hand, if $\delta$ branches off at the the
  first point of $(\gamma)\tau_m$, then the first arrow on $L$ is not
  the last arrow on $\gamma$, and so again the second point of
  $\delta$ lies outside $L$. Therefore we obtain a cycle by following
  the path $\delta$ to $y$, then the path in $L$ back to the start of
  the infinite track, then back along the infinite path to the start
  of $\delta$. This is a contradiction, and so $\delta$ cannot visit
  any point of $L$.

\item Suppose that for some $j > -h$, the path $(\delta)\tau_{(h+j)m}$
  is not defined. Since $\delta$ consists only of short arrows, we can
  use the reasoning in the proof of
  \fullref{Lemma}{lem:shallowstarpathprops} to take a shorter path
  $\delta'$ such that $(\delta')\tau_{(h+j)m}$ is defined, present in
  the diagram, and ends at some point in $C[1] \subseteq L$. (See
  \fullref{Figure}{fig:periodicpathpathprops}.) Since $\delta$ does
  not include any edge of any $(\gamma)\tau_{hm}$ or $\gamma$, neither
  does $\delta'$. Hence $(\delta')\tau_{(h+j)m}$ does not include any
  edge of $\gamma$ or $(\gamma)\tau_m$. This contradicts part~1
  applied to $(\delta')\tau_{(h+j)m}$.

\begin{figure}[tb]
\centerline{\includegraphics{\jobname-tree-periodicpathpathprops.eps}}
\caption{We suppose there is some path
  $\delta$ and $j> -h$ such that $(\delta)\tau_{(h+j)m}$
  is undefined. Then $\tau_{(h+j)m}$ shifts part of $\delta$ off the
  left-hand side of the diagram. Since $\delta$ consists only of short
  arrows, we can choose $\delta'$ to be that initial part of $\delta$
  that is shifted by $\tau_{(h+j)m}$ to end in $C[1]$.}
\label{fig:periodicpathpathprops}
\end{figure}

\item Suppose $\delta$ has length greater than $D^2+D$. Let $\delta$
  branch off from the infinite periodic track at a point $q$. Since
  $\delta$ includes no point from $C[0] \subseteq L$, the path
  $(\delta)\tau_{jm}$ is defined for all $j \in \nset^0$. Since the
  infinite track is made up of $(\gamma)\tau_m$,
  $(\gamma)\tau_{2m}$,\ldots, it follows that $q\tau_{jm}$ lies on the
  infinite track for all $j \in \nset^0$. Then since $\delta$ has
  length greater than $D^2+D$ and there are only $D$ rows and $m \leq
  D+1$ there exist two points $x$ and $y$ on $\delta$ such that $b(x)
  = b(y)$ and $c(x) \equiv c(y) \pmod m$. Without loss of generality,
  suppose $c(x) < c(y)$. Let $j = (c(y)-c(x))/m$. Then $x\tau_{jm} =
  y$, and so $y$ also lies on the path $(\delta)\tau_{jm}$, which is
  defined and present in the diagram by \fullref{Arrow
    rule}{arrowrule:up}. Let $\alpha_1$ be the subpath of $\delta$
  from $q$ to $y$; let $\alpha_2$ be the subpath of
  $(\delta)\tau_{jm}$ from $q\tau_{jm}$ to $y$. Then $\alpha_1$,
  $\alpha_2$, and the part of the infinite periodic track between $q$
  and $q\tau_{jm}$ form a cycle, which is a contradiction. So $\delta$
  cannot have length greater than $D^2+D$.

\end{enumerate}
\end{proof}

For $i \in \nset$, let $F_i$ consist of all the points connected to
$(\gamma)\tau_{im}$ by any simple path that does not include any edges
of $(\gamma)\tau_{im}$ or $(\gamma)\tau_{(i-1)m}$. (That is, $F_i$
consists of the points lying on paths that branch off from the
infinite periodic track at some point on $(\gamma)\tau_{im}$. Notice
that $F_i$ contains all points of $(\gamma)\tau_{im}$.) By
\fullref{Lemma}{lem:periodicpathprops}(3), each such path is of
bounded length. Since the graph is of bounded degree, each $F_i$ is
finite. Furthermore, by \fullref{Lemma}{lem:periodicpathprops}(3),
$F_i\tau_{jm} = F_{i+j}$ (where $i \in \nset$ and $j \in \zset$ with $-i
< j$), and the subgraphs induced by each $F_i\phi$ are isomorphic.

Let $(S,\sigma)$ be the subgraph induced by the vertices $F_1\phi$. Let
$s_0$ be the vertex corresponding to the first vertex on
$(\gamma)\tau_m$; Let $s_1$ be the vertex corresponding to the last
vertex on $(\gamma)\tau_m$. Then since the subgraphs induced by the
$F_i\phi$ are isomorphic, the infinite periodic track
$(\gamma)\tau_m,(\gamma)\tau_{2m},\ldots$ is mapped by $\phi$ to the
spine of the infinite periodic path $\mathcal{P}(S,\sigma,s_0,s_1)$. It
is clear that $\mathcal{P}(S,\sigma,s_0,s_1)$ is attached at its base to
some vertex in the rest of the graph.

It has already been established that there are only finitely many
infinite periodic paths, so the graph must be made up of finitely many
periodic paths attached to the finite subgraph induced
by $L'\phi$. This completes the proof in this direction.
\end{proof}

Equipped with a characterization of unary FA-pre\-sent\-a\-ble directed
trees, we now turn to characterizing the unary FA-pre\-sent\-a\-ble directed
forests:

\begin{theorem}
\label{thm:directedforestchar}
A countable directed forest is unary FA-pre\-sent\-a\-ble if and only if:
\begin{enumerate}
\item It has only finitely many infinite components, each of which is
  a unary FA-pre\-sent\-a\-ble directed tree.
\item There is a bound on the size of its finite
  components.
\end{enumerate}
\end{theorem}

\begin{proof}
Let $(a^*,\phi)$ be an injective unary FA-pre\-sent\-a\-tion for a directed
forest $(T,\eta)$. Let $\zeta$ be the equivalence relation generated by
$\eta$; then $\Lambda(\zeta,\phi)$ is regular by
\fullref{Corollary}{corol:equivrelgen} and so $(a^*,\phi)$ is a unary
FA-pre\-sent\-a\-tion for $(T,\eta,\zeta)$. Notice that $\zeta$ is the
undirected reachability relation on $(T,\eta)$ and so its equivalence
classes are the connected components of $(T,\eta)$, which are directed
trees. By \fullref{Theorem}{thm:equivrelchar}, there are finitely many
infinite components and a bound on the cardinality of the finite
components. Consider some infinite component $U$. Notice that the set
of elements in $U$ first-order definable in terms of a $\zeta$ and
some $u \in U$. So the language $K$ of words in $a^*$ that represent
elements of $U$ (that is, $K = U\phi^{-1}$) is regular. Thus
$(K,\phi|_K)$ is a unary FA-pre\-sent\-a\-tion for the component $U$. Since
$U$ was arbitrary, every infinite component is unary
FA-pre\-sent\-a\-ble. This completes one direction of the proof.

Let $(T,\eta)$ be a countable directed forest that has only finitely many
infinite components, each of which is a unary FA-pre\-sent\-a\-ble directed
tree, and with a bound on the size of its finite
components.

Consider the finite components. Since each is a directed tree and
there is a bound on their cardinalities, there are only finitely many
isomorphism types amongst them. Let $(P_1,\pi_1),\ldots,(P_p,\pi_p)$
be those finite components whose isomorphism types appear only
finitely many times among the finite components of $(T,\eta)$. Suppose
there are $q$ different isomorphism types that appear infinitely
often. For $i \in \{1,\ldots,q\}$, choose a represenative
$(Q_i,\kappa_i)$ of each isomorphism class. Let
$(R_1,\rho_1),\ldots,(R_r,\rho_r)$ be the infinite components.

For each $i$, the union of countably many copies of the finite
directed tree $(Q_i,\kappa_i)$ is unary FA-pre\-sent\-a\-ble by
\fullref{Lemma}{lem:countableunion}. The union $(Q,\kappa)$ of the $q$
forests thus obtained is unary FA-pre\-sent\-a\-ble by iterated application
of \fullref{Lemma}{lem:union}. Thus the directed forest $(T,\eta)$,
which is the union of $(Q,\kappa)$ and the various $(P_i,\pi_i)$ and
$(R_i,\rho_i)$ is unary FA-pre\-sent\-a\-ble by  iterated application
of \fullref{Lemma}{lem:union}.
\end{proof}

Finally, we can apply the characterization of unary FA-pre\-sent\-a\-ble
directed forests to obtain a characterization of unary FA-pre\-sent\-a\-ble
[undirected] forests:

\begin{theorem}
\label{thm:forestchar}
A forest is unary FA-pre\-sent\-a\-ble if and only if it can be obtained
from a unary FA-pre\-sent\-a\-ble directed forest by changing directed edges
to undirected edges (that is, by replacing the edge relation with its
symmetric closure).
\end{theorem}

\begin{proof}
First, notice that if $(T,\eta)$ is a unary FA-pre\-sent\-a\-ble directed
forest, then the symmetric closure $\sigma$ of $\eta$ is first-order
definable in terms of $\eta$. Thus $(T,\sigma)$ is also unary
FA-pre\-sent\-a\-ble.

Let $(T,\eta)$ be forest admitting an injective unary FA-pre\-sent\-a\-tion
$(a^*,\phi)$. The edge relation $\eta$ is symmetric. Define a new
relation $\eta'$ as follows
\[
(s,t) \eta' \iff \bigl((s,t) \in \eta\bigr) \land \bigl(\ell(s) < \ell(t)\bigr).
\]
Then $(T,\eta')$ is a directed graph. For every pair of elements $s,t$
that are connected by an (undirected) edge in $(T,\eta)$ (that is,
both $(s,t)$ and $(t,s)$ are in $\eta$), exactly one of $\ell(s) <
\ell(t)$ or $\ell(t) < \ell(s)$ holds, and thus there is either a
directed edge from $s$ to $t$ in $(T,\eta')$ (that is, $(s,t) \in
\eta'$) or an edge from $t$ to $s$ in $(T,\eta')$ (that is $(t,s) \in
\eta'$). So $(T,\eta')$ is an undirected forest. Furthermore,
$\Lambda(\eta',\phi)$ is regular since a finite automaton can compare
the lengths of its two input words. Thus $(T,\eta')$ is unary
FA-pre\-sent\-a\-ble. It is clear that making $\eta'$ symmetric yields
$(T,\eta)$.
\end{proof}

\section{Maps and partial maps}

In this final section, we apply the results of previous sections,
particularly \fullref{\S}{sec:trees}, to classify the orbit structures
of unary FA-pre\-sent\-a\-ble maps and more generally partial maps.

The graph of a partial map $f : X \to X$ is the directed graph with
vertex set $X$ and edge set $\{(x,(x)f) : x \in X, \text{$(x)f$ is
  defined}\}$. Two elements $x,y \in X$ lie in the same orbit if there
exist $m,n \in \nset^0$ such that $(x)f^m = (y)f^n$. They lie in the
same strong orbit if there exist $m,n \in \nset^0$ such that $(x)f^m =
y$ and $(y)f^n = x$. In terms of the graph, $x$ and $y$ lie in the
same orbit if they are connected by an undirected path; $x$ and $y$
lie in the same strong orbit if there is a directed path from $x$ to
$y$ and a directed path from $y$ to $x$. Hence the orbits of the
partial map are the connected components of the graph; the strong
orbits of the partial map are the strongly connected components of the
graph. Our characterization results are all stated in these terms.

We start by characterizing unary FA-presentable maps and partial
(\fullref{Theorem}{thm:mapchar}). Starting from this result, we then
obtain characterizations of unary FA-presentable injections and
partial injections (\fullref{Theorem}{thm:injectionchar}), surjections
and partial surjections (\fullref{Theorem}{thm:surjectionchar}), and
bijections and partial bijections
(\fullref{Theorem}{thm:bijectionchar}). Naturally, the
characterization for bijections is equivalent to the previously known
one \cite[Theorem~7.12]{blumensath_diploma}, albeit in a very
different form.

Like the characterization result for directed trees
(\fullref{Theorem}{thm:directedtreechar}), the characterization
results for orbit structures of unary FA-pre\-sent\-a\-ble partial
maps are stated in terms of attaching periodic paths and shallow stars
to finite graphs. We need to specify certain special types of periodic
paths and shallow stars, and also define some related terms:

\begin{definition}
Retain notation from
\fullref{Definition}{def:shallowstarperiodicpath}.

An \defterm{inward rooted tree} is a tree with a distinguished vertex,
called the \defterm{root}, towards which all its edges are oriented.

A shallow star $\mathcal{S}(T,\eta,t_0,t_1)$ is \defterm{inwardly
  oriented} if every edge of the template graph $(T,\eta,t_0,t_1)$ is
oriented towards the distinguished vertex $t_0 = t_1$. That is, every edge
of $\mathcal{S}(T,\eta,t_0,t_1)$ is oriented towards the centre vertex.

A periodic path $\mathcal{S}(T,\eta,t_0,t_1)$ is
\defterm{inwardly oriented} if every edge of the template graph
$(T,\eta,t_0,t_1)$ is oriented towards the distinguished vertex
$t_0$. That is, every edge of $\mathcal{P}(T,\eta,t_0,t_1)$ is oriented
towards the base.

A periodic path $\mathcal{S}(T,\eta,t_0,t_1)$ is \defterm{outwardly
  oriented} if every edge of the template graph $(T,\eta,t_0,t_1)$ is
oriented towards the distinguished vertex $t_1$. That is, every edge
of $\mathcal{P}(T,\eta,t_0,t_1)$ is oriented towards the unbounded
direction of the spine.

An \defterm{inward path} is the periodic path
$\mathcal{P}(\{t_0,t_1\},\{(t_1,t_0)\},t_0,t_1)$ (where the template
has two vertices $t_0$ and $t_1$ and a single edge from $t_1$ to
$t_0$). Notice that an inward path is an inwardly oriented periodic
path.

An \defterm{outward path} is the periodic path
$\mathcal{P}(\{t_0,t_1\},\{(t_0,t_1)\},t_0,t_1)$ (where the template
has two vertices $t_0$ and $t_1$ and a single edge from $t_0$ to
$t_1$). Notice that an outward path is an outwardly oriented periodic
path.

A \defterm{bi-infinite path} is a directed path with vertex set $\{v_i
: i \in \zset\}$ and edge relation $\{(v_i,v_{i+1}) : i \in \zset\}$;
this is isomorphic to the path obtained by attaching an outward path
to an inward path at their base vertices.
\end{definition}

\begin{theorem}
\label{thm:unarymapsingle}
A map with a single orbit is unary FA-pre\-sent\-a\-ble if and only
if its orbit can be obtained in one of the following ways:
\begin{enumerate}

\item Start with a finite directed cycle. First, attach finitely many
  inward rooted finite trees (at their roots) to the cycle. To any
  vertices of the resulting finite graph, attach finitely many
  inwardly oriented periodic paths (at their bases) and
  finitely many inwardly oriented shallow stars (at their centres).

\item Start with an inward rooted finite tree. To the root of the tree,
  attach one outwardly oriented periodic path (at its base). To
  any vertices of the resulting graph, attach finitely many inwardly
  oriented periodic paths (at their bases) and finitely many inwardly
  oriented shallow stars (at their centres).

\end{enumerate}
\end{theorem}

[\fullref{Figure}{fig:maporbits} shows some examples of orbits
  described in \fullref{Theorem}{thm:unarypartialmapsingle}. Notice that if
  no periodic paths or shallow stars are attached case~1 gives a finite
  graph. Case~2 requires that an outwardly oriented periodic path is
  attached and so always yields an infinite graph.]

\begin{figure}[t]
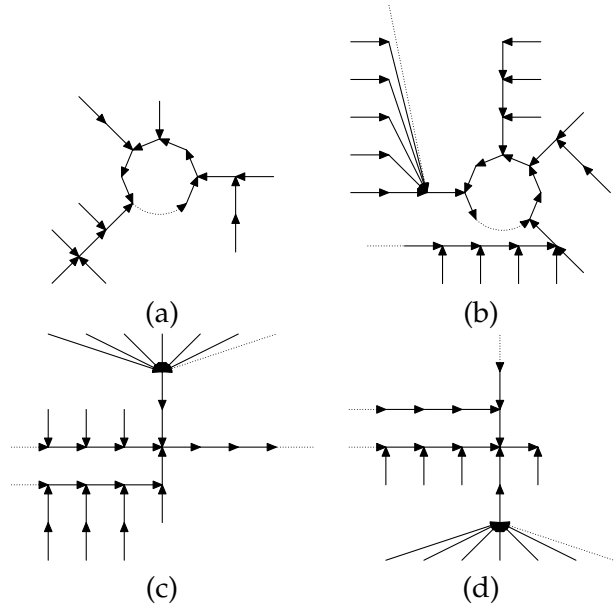

\centerline{%
\begin{tabular}{cc}
\includegraphics{\jobname-map-general1.eps} &
\includegraphics{\jobname-map-general2.eps} \\
(a) & (b) \\
\includegraphics{\jobname-map-general3.eps} &
\includegraphics{\jobname-map-general4.eps} \\
(c) & (d)
\end{tabular}%
}
\caption{Possible example orbits of a unary FA-pre\-sent\-a\-ble maps
  and partial maps: (a) is finite, (b) contains a cycle (arising from
  case~1 in \fullref{Theorem}{thm:unarymapsingle}), (c) contains an
  infinite outwardly oriented periodic path (arising from case~2 in
  \fullref{Theorem}{thm:unarymapsingle}), and (d) contains a vertex of
  outdegree $0$ (where the map is undefined, arising from case~3 in
  \fullref{Theorem}{thm:unarypartialmapsingle}).}
\label{fig:maporbits}
\end{figure}

\begin{proof}

\emph{First part.} Let $(a^*,\phi)$ be a unary FA-pre\-sent\-a\-tion for
$(X,f)$, where $f : X \to X$ has only one orbit. Let $Q$ be the
transitive closure of $f$; then $\Lambda(Q,\phi)$ is
regular by \fullref{Theorem}{thm:transclosure}. Define a relation
$R$ on $X$ by
\[
(x,y) \in R \iff \bigl((x,y) \in Q\bigr) \land \bigl((y,x) \in Q\bigr).
\]
Then two distinct elements of $X$ are related by $R$ if and only if
they lie in the same strong orbit of $f$. Notice that $(x,x) \in R$ if
and only if $(x)f^m = x$ for some $m > 0$. Thus the strong orbits of
$f$ are the $R$-classes plus singleton orbits for all elements not in
some $R$-class.

Notice that any $R$-classes must form a directed cycle. There cannot
be two distinct $R$-classes, for otherwise they would be connected by
some path, and then some vertex on this path would have outdegree $2$,
which is impossible. So there is either a unique $R$-class or no
$R$-class. Deal with these cases separately.

\begin{enumerate}

\item There is a unique $R$-class. Since $\Lambda(R,\phi)$ is regular,
  we can factor $(X,f)$ by $R$ to get a unary FA-pre\-sent\-a\-ble map
  $(X,f')$, which is essentially the map $(X,f)$ with all the points
  of the $R$-class merged to a single point $z$. Notice that $(z)f'
  = z$.

  Consider the graph $(X',f')$. Remove the single edge from $z$ to
  itself. The resulting graph is a directed tree (with all edges
  oriented towards $z$) and is unary FA-pre\-sent\-a\-ble. So by
  \fullref{Theorem}{thm:directedtreechar}, this tree consists of a
  finite tree with finitely many periodic paths and shallow stars
  attached. Since all the edges are oriented towards $z$, this element
  $z$ must lie in the finite graph (since the orientation of edges
  along the spine of a periodic paths is periodic). So all the
  periodic decorated paths and all the shallow stars are inwardly
  oriented.

  Thus the original graph $(X,f)$ must consist of these same inwardly
  oriented periodic paths and inwardly oriented shallow stars attached
  to finite inward rooted trees, which are in turn attached at their
  roots to the finite cycle that forms the unique $R$-class of
  $f$. Thus case~1 in the statement of the theorem holds.

\item There is no $R$-class. Then the graph $(X,f)$ is a tree and so
  by \fullref{Theorem}{thm:directedtreechar}, this tree consists of a
  finite tree $F$ with finitely many periodic paths and shallow stars
  attached. Since all vertices have outdegree $1$, all the shallow
  stars are inwardly oriented. Since every vertex has outdegree $1$,
  all edges of $F$ must be oriented towards a particular vertex, where
  a single outwardly oriented periodic path must be attached. All the
  other periodic paths attached must be inwardly oriented, again by
  the fact that all vertices have outdegree $1$. Thus case~2 in the
  statement of the theorem holds.

\end{enumerate}

\smallskip
\noindent\emph{Second part.} If $(X,\eta)$ is a graph as described in
the theorem statement, then every vertex of $(X,\eta)$ has outdegree
exactly $1$. Thus we can define a map $f : X \to X$ by letting $(x)f$
be the terminal vertex of the unique edge starting at $x$. It is clear
that $(X,\eta)$ is the graph of $(X,f)$. Furthermore, $(X,\eta)$, and
hence $(X,f)$ is unary FA-pre\-sent\-a\-ble by
\fullref{Lemmata}{lem:attachmentunaryfa} and
\ref{lem:shallowstarperiodicpathunaryfa}.
\end{proof}

\begin{theorem}
\label{thm:unarypartialmapsingle}
A partial map with a single orbit is unary FA-pre\-sent\-a\-ble if and
only if its orbit can be obtained in as described in case~1 or~2 of
\fullref{Theorem}{thm:unarymapsingle} or in the following way:
\begin{enumerate}
\item[3.] Start with an inward rooted finite tree. To any vertices of
  this finite graph, attach finitely many inwardly oriented periodic
  paths (at their bases) and finitely many inwardly oriented shallow
  stars (at their centres).
\end{enumerate}
\end{theorem}

\begin{proof}
\emph{First part.} Let $(a^*,\phi)$ be a unary FA-pre\-sent\-a\-tion
for $(X,f)$, where $f : X \to X$ has only one orbit. Extend $f$ to a
complete map $f' : X \to X $ by defining
\[
(x)f' = \begin{cases} (x)f & \text{if $(x)f$ is defined} \\
x & \text{otherwise.}
\end{cases}
\]
From the graph perspective $(X,f')$ is formed by taking the graph
$(X,f)$ and adding a loop at every vertex of outdegree $0$.

Notice that the support of $f$ is first-order definable and so
$(X,f')$ is also unary FA-presentable. Furthermore, $(X,f')$ also has
only one orbit. So the graph $(X,f')$ is as described in
\fullref{Theorem}{thm:unarymapsingle}. If $(X,f)$ and $(X,f')$ are
identical, the proof is complete. So assume that $(X,f)$ and $(X,f')$
are distinct. Then at least one loop is added to the graph $(X,f)$ to
form $(X,f')$. So case~1 of \fullref{Theorem}{thm:unarymapsingle}
applies, with the initial cycle being a loop at a single
vertex. Removing this (unique) loop to recover $(X,f)$ gives a graph
obtained as described in case~3 of the theorem statement.

\medskip
\noindent\emph{Second part.} If $(X,\eta)$ is a graph as described in
the statement, then every vertex of $(X,\eta)$ has outdegree $0$ or
$1$. Thus we can define a partial map $f : X \to X$ by letting $(x)f$
be the terminal vertex of the unique edge starting at $x$, if such an
edge exists, and otherwise leaving $(x)f$ undefined. It is clear that
$(X,\eta)$ is the graph of $(X,f)$. Furthermore, $(X,\eta)$, and hence
$(X,f)$ is unary FA-pre\-sent\-a\-ble by
\fullref{Lemmata}{lem:attachmentunaryfa} and
\ref{lem:shallowstarperiodicpathunaryfa}.
\end{proof}

With this characterization of individual orbits, the characterization
of the orbit structures of unary FA-pre\-sent\-a\-le maps now follows
quickly:

\begin{theorem}
\label{thm:mapchar}
A map (respectively, partial map) is unary FA-pre\-sent\-a\-ble if and
only if the following conditions hold:
\begin{enumerate}

\item There is a bound on the size of the finite orbits.

\item There are finitely many infinite orbits.

\item Each orbit is unary FA-pre\-sent\-a\-ble and so as described in
  \fullref{Theorem}{thm:unarymapsingle} (respectively, \fullref{Theorem}{thm:unarypartialmapsingle}).

\end{enumerate}
\end{theorem}

\begin{proof}
Let $(X,f : X \to X)$ be a unary FA-pre\-sent\-a\-ble map
(respectively, partial map). If $X$ is finite, there is nothing to
prove. So assume $X$ is infinite and let $(a^*,\phi)$ be a unary
FA-pre\-sent\-a\-tion for $(X,f)$. Let $Q$ be the equivalence relation
generated by $f$. Since $\Lambda(Q,\phi)$ is regular by
\fullref{Corollary}{corol:equivrelgen}, $Q$ must have finitely many
infinite equivalence classes and a bound on the size of its finite
equivalence classes by \fullref{Theorem}{thm:equivrelchar}. But the
equivalence classes are simply the orbits of $f$. It remains to
observe that since the membership relation of each of the equivalence
classes is first-order definable, the set of words representing
elements of any orbit is regular, and thus the map (respectively,
partial map) $f$ restricted to any orbit is unary
FA-pre\-sent\-a\-ble, and hence the restriction of $f$ to each such
infinite orbit is thus as desribed in
\fullref{Theorem}{thm:unarymapsingle} (respectively,
\fullref{Theorem}{thm:unarypartialmapsingle}).

In the other direction, the result follows by applying
\fullref{Lemmata}{lem:union} and \ref{lem:countableunion} in a manner
similar to the proof of \fullref{Theorem}{thm:directedforestchar}.
\end{proof}

We can now characterize unary FA-pre\-sent\-a\-ble surjections, injections,
and bijections.

\begin{theorem}
\label{thm:surjectionchar}
A surjective map is unary FA-pre\-sent\-a\-ble if and only if the following
conditions hold:
\begin{enumerate}

\item There is a bound on the size of the finite orbits, and every
  finite orbit is a cycle.

\item There are finitely many infinite orbits, and each can be
  obtained in one of two ways:

\begin{enumerate}

\item Start with a finite directed cycle. First, attach finitely many
  inward rooted finite trees (at their roots) to the cycle. To every
  leaf vertex, and possibly to other vertices, of the resulting finite
  graph, attach finitely many inward infinite paths (at their bases).

\item Start with an inward rooted finite tree. To the root of the tree,
  attach root one outwardly oriented infinite path (at its base). To
  every leaf vertex, and possibly to other vertices, of the resulting
  graph, attach finitely many inwardly oriented periodic paths (at
  their bases).

\end{enumerate}
\end{enumerate}
A partial surjective map is unary FA-pre\-sent\-a\-ble if and only if
the following conditions hold: its finite orbits are as described in
condition~1, and its infinite orbit are obtained either as
described in case~2(a) or~2(b) above or in the following way:
\begin{enumerate}

\item[2.]
\begin{enumerate}

\item[(c)] Start with an inward rooted finite tree. To every leaf vertex,
  and possibly to other vertices, of the resulting graph, attach
  finitely many inwardly oriented periodic paths (at their bases).

\end{enumerate}
\end{enumerate}

\end{theorem}

\begin{figure}[t]
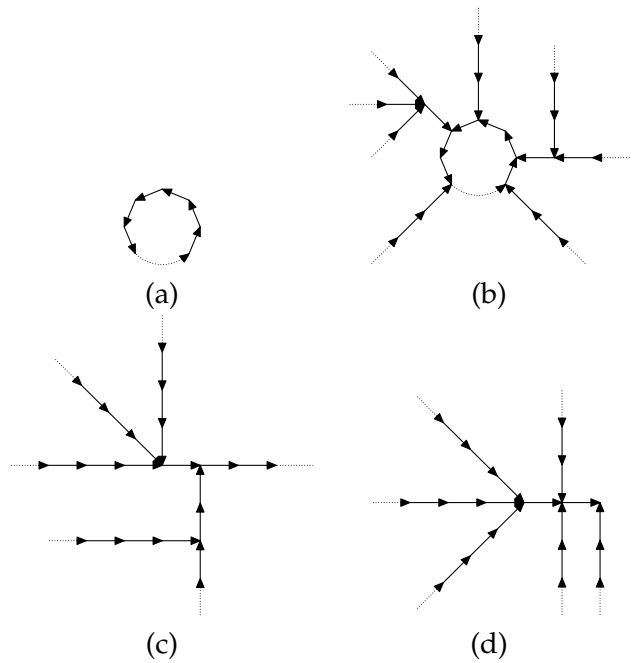

\centerline{%
\begin{tabular}{cc}
\includegraphics{\jobname-map-cycle.eps} &
\includegraphics{\jobname-map-surjection2.eps} \\
(a) & (b) \\
\includegraphics{\jobname-map-surjection3.eps} &
\includegraphics{\jobname-map-surjection4.eps} \\
(c) & (d)
\end{tabular}%
}
\caption{Possible example orbits of unary FA-pre\-sent\-a\-ble
  surjections and partial surjections: (a) is finite, (b) contains a cycle (arising from
  case~2(a) in \fullref{Theorem}{thm:surjectionchar}), (c) contains an
  infinite outwardly oriented periodic path (arising from case~2(b) in
  \fullref{Theorem}{thm:surjectionchar}), and (d) contains a vertex
  with outdegree $0$ (where the map is undefined arising from
  case~2(c) in \fullref{Theorem}{thm:surjectionchar}).}
\label{fig:surjectionorbits}
\end{figure}

\begin{proof}
\textit{Complete surjective maps.} Let $(X,f : X \to X)$ be a unary
FA-pre\-sent\-a\-ble surjective map. Then its orbits are as described
in \fullref{Theorems}{thm:mapchar} and \ref{thm:unarymapsingle}. Every
vertex of the graph $(X,f)$ has indegree at least $1$. Every finite
orbit must therefore be a cycle.

Consider some infinite orbit. Suppose first that this orbit contains a
cycle (case~1 of \fullref{Theorem}{thm:unarymapsingle}). Every vertex
that does not lie on this cycle must lie on an infinite inward path,
by induction using the fact that every vertex has indegree at least
one. Thus there can be no shallow stars attached, all attached
periodic paths must be inward infinite paths, and at least one inward
infinite path must be attached to every leaf vertex of the finite
graph.

Suppose now that this orbit does not contain a cycle (case~2 of
\fullref{Theorem}{thm:unarymapsingle}). Then every vertex must lie on
an infinite inward path, by induction using the fact that every vertex
has indegree at least one. Thus there can be no shallow stars
attached, a single outward infinite path must be attached to the root
of the finitre tree, all other attached periodic paths paths must be
inward infinite path, and at least one inward infinite path must be
attached to every leaf vertex of the initial finite tree.

In the other direction, a graph of the form described in the statement
is the graph of a unary FA-pre\-sent\-a\-ble map by
\fullref{Theorem}{thm:mapchar}. Furthermore, every vertex of such a
graph has outdegree exactly $1$ and indegree at least $1$ and hence is
the graph of a surjection.

\medskip
\textit{Partial surjective maps.} The strategy is essentially the same
as the proof \fullref{Theorem}{thm:unarypartialmapsingle}, so we only
sketch the proof. Let $(X,f : X \to X)$ be a unary
FA-pre\-sent\-a\-ble surjective partial map. Extend the map to a
complete map $f'$ by defining $(x)f' = x$ whenever $(x)f$ is
undefined. Note that this preserves unary FA-presentability and
surjectivity. Any infinite orbit where $f'$ does not coincide with $f$
contains a loop and so case~2(a) applies. Removing this loop yields a
graph that can be obtained as per case~2(c).
\end{proof}

\begin{theorem}
\label{thm:injectionchar}
An injective map is unary FA-pre\-sent\-a\-ble if and only if its orbits
satisfy the following conditions:
\begin{enumerate}

\item There is a bound on the size of the finite orbits, and every
  finite orbit is a cycle.

\item There are finitely many infinite orbits, each being either an
  outwardly oriented infinite path or a bi-infinite path.

\end{enumerate}
A partial injective map is unary FA-pre\-sent\-a\-ble if and only if its orbits
satisfy the following conditions:
\begin{enumerate}

\item[3.] There is a bound on the size of the finite orbits, and every
  finite orbit is a cycle or a finite path.

\item[4.] There are finitely many infinite orbits, each being either an
  outwardly oriented infinite path, an inwardly oriented infinite path
  or a bi-infinite path.

\end{enumerate}
\end{theorem}

\begin{figure}[t]
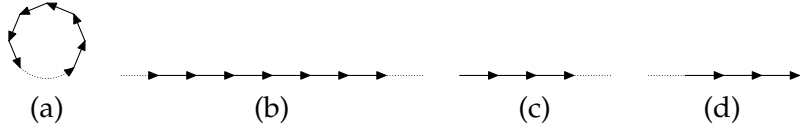

\centerline{%
\begin{tabular}{cccc}
\includegraphics{\jobname-map-cycle.eps} &
\includegraphics{\jobname-map-bipath.eps} &
\includegraphics{\jobname-map-outpath.eps} &
\includegraphics{\jobname-map-inpath.eps} \\
(a) & (b) & (c) & (d)
\end{tabular}%
}
\caption{Possible example orbits of unary FA-pre\-sent\-a\-ble
  injections and partial injections: (a) is a finite cycle, (b) is a
  bi-infinite path, (c) is an outward infinite path, and (d) is an
  inward infinite path. All three of (a), (b), and (c) can arise in
  (complete) injections and partial injections, but (d) can only arise
  in partial injections.}
\label{fig:injectionorbits}
\end{figure}

\begin{proof}
Let $(X,f : X \to X)$ be a unary FA-pre\-sent\-a\-ble injective map
(respectively, partial injective map). Then its orbits are as
described in \fullref{Theorems}{thm:mapchar} and
\ref{thm:unarymapsingle} (respectively, \fullref{Theorems}{thm:mapchar} and
\ref{thm:unarypartialmapsingle}). In particular, there is a bound on the size
of the finite orbits and finitely many infinite orbits.

Now, every vertex of the graph $(X,f)$ has outdegree $1$
(respectively, at most $1$) and indegree at most $1$. It follows that
every finite orbit must be a cycle (respectively, a cycle or a finite
path), and every infinite orbit either an outward (respectively,
outward or inward) infinite path or a bi-infinite path.

In the other direction, a graph of the form described in the statement
is the graph of a unary FA-pre\-sent\-a\-ble map (respectively,
partial map) by \fullref{Theorem}{thm:mapchar}. Furthermore, every
vertex of such a graph has outdegree $1$ (respectively, at most $1$)
and indegree at most $1$ and hence is the graph of a injection
(respectively, partial injection).
\end{proof}

\begin{theorem}
\label{thm:bijectionchar}
A bijective map is unary FA-pre\-sent\-a\-ble if and only if its orbits
satisfy the following conditions:
\begin{enumerate}

\item There is a bound on the size of the finite orbits, and every
  finite orbit is a cycle.

\item There are finitely many infinite orbits, each being a bi-infinite path.

\end{enumerate}
A partial bijection is unary FA-pre\-sent\-a\-ble if and only if its orbits
satisfy the following conditions:
\begin{enumerate}

\item There is a bound on the size of the finite orbits, and every
  finite orbit is a cycle.

\item There are finitely many infinite orbits, each being an inward
  infinite path or a bi-infinite path.

\end{enumerate}
\end{theorem}

\begin{figure}[t]
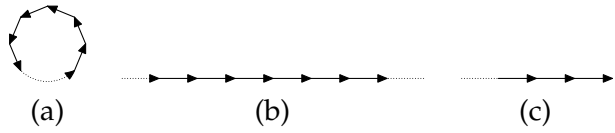

\centerline{%
\begin{tabular}{ccc}
\includegraphics{\jobname-map-cycle.eps} &
\includegraphics{\jobname-map-bipath.eps} &
\includegraphics{\jobname-map-inpath.eps} \\
(a) & (b) & (c)
\end{tabular}%
}
\caption{Possible example orbits of unary FA-pre\-sent\-a\-ble
  bijections and partial bijections: (a) is a finite cycle, (b) is a
  bi-infinite path, and (c) is an inward infinite path. Both (a) and
  (b) can arise in (complete) bijections and partial bijections, but
  (c) only in partial bijections.}
\label{fig:bijectionorbits}
\end{figure}

\begin{proof}
Let $(X,f : X \to X)$ be a unary FA-pre\-sent\-a\-ble bijection
(respectively, partial bijection). Then in particular $f$ is injective
and so its orbits are as described in
\fullref{Theorem}{thm:injectionchar}. Every vertex of the graph
$(X,f)$ has outdegree $1$ (respectively, at most $1$) and indegree $1$
and no infinite orbit can consist of a outward infinite path. Every
infinite orbit is thus a bi-infinite path (respectively, an inward
infinite path or a bi-infinite path).

In the other direction, a graph of the form described in the statement
is the graph of a unary FA-pre\-sent\-a\-ble map (respectively,
partial map) by \fullref{Theorem}{thm:mapchar}. Furthermore, every
vertex of such a graph has outdegree exactly $1$ (respectively, at
most $1$) and indegree exactly $1$ and hence is the graph of a
bijection (respectively, partial bijection).
\end{proof}


\end{document}